\documentclass[11pt]{amsart}       
\usepackage{graphicx}
\usepackage{latexsym}
\usepackage{enumerate}
\usepackage{amsmath}
\usepackage{amssymb}
\usepackage{amsfonts}
\usepackage{multirow}
\usepackage{slashbox}
\usepackage{multirow}
\usepackage[bookmarks,pagebackref]{hyperref}
\numberwithin{equation}{section}

\newcommand{\A}{\mathcal{A}}
\newcommand{\Au}{\mathrm{Aut}\,}
\newcommand{\ad}{\mathrm{ad}}
\newcommand{\af}{\mathrm{aff}}

\newcommand{\C}{\mathbb{C}}

\newcommand{\di}{\mathrm{diag}}
\newcommand{\E}{\mathrm{End}\,}
\newcommand{\F}{\mathbb{F}}

\newcommand{\G}{\mathcal{G}}
\newcommand{\GL}{\mathrm{GL}}
\newcommand{\g}{\mathfrak{g}}
\newcommand{\h}{\mathfrak{h}}
\newcommand{\I}{\mathrm{Im} \,}
\newcommand{\Ker}{\mathrm{Ker}}
\newcommand{\K}{\mathbb{K}}
\newcommand{\Li}{$\mathrm{Lie} \left(n, 2\right)$}
\newcommand{\li}{$\mathrm{Lie} \left(n, n-2\right)$}
\newcommand{\M}{\mathrm{Mat}}

\newcommand{\R}{\mathbb{R}}
\newcommand{\ra}{\mathrm{rank}}
\newcommand{\s}{\mathrm{Span}}

\newcommand{\T}{\mathrm{Tr} \,}
\newcommand{\Z}{\mathcal{Z}}

\newtheorem{thm}{Theorem}[section]
\newtheorem{lem}[thm]{Lemma}
\newtheorem{prop}[thm]{Proposition}
\newtheorem{cor}[thm]{Corollary}
\newtheorem{rem}[thm]{Remark}

\theoremstyle{definition}
\newtheorem{defn}[thm]{Definition}
\newtheorem*{aknow}{Acknowledgments}

\usepackage[top=3cm,bottom=3cm,left=2.5cm,right=2.5cm]{geometry}

\begin{document}

\title[Solvable Lie algebras having 2-dimensional or 2-codimensional derived ideal]{On the classifying problem for the class of real solvable Lie algebras having 2-dimensional or 2-codimensional derived ideal
}


\author{Vu A. Le \and Tuan A. Nguyen \and  Tu T. C. Nguyen \and Tuyen T. M. Nguyen \and Hoa Q. Duong}


\address{Vu A. Le, Department of Economic Mathematics, University of Economics and Law, Vietnam National University - Ho Chi Minh City, Vietnam.}
\email{vula@uel.edu.vn}

\address{Tuan A. Nguyen, Faculty of Political Science and Pedagogy, University of Physical Education and Sports, Ho Chi Minh City, Vietnam.}
\email{natuan@upes.edu.vn}

\address{Tu T. C. Nguyen, College of Natural Sciences, Can Tho University, Can Tho City, Vietnam.}
\email{camtu@ctu.edu.vn}

\address{Tuyen T. M. Nguyen, Faculty of Mathematics and Information, Dong Thap University, Cao Lanh city, Dong Thap Province, Vietnam.}
\email{ntmtuyen@dthu.edu.vn}

\address{Hoa Q. Duong, General Education Program, Hoa Sen University, Ho Chi Minh City, Vietnam.}
\email{hoa.duongquang@hoasen.edu.vn}


\keywords{Lie algebras; Derived ideals; {\Li}-algebras; {\li}-algebras}

\subjclass[2000]{Primary 17B, 22E60, Secondary 20G05.}

\maketitle

	\begin{abstract}
	   Let $\mathrm{Lie} \left(n, k\right)$ denote the class of all $n$-dimensional real solvable Lie algebras having $k$-dimensional 
	   derived ideal ($1 \leqslant k \leqslant n-1$). In 1993, the class $\mathrm{Lie} \left(n, 1\right)$ was completely classified by 
	   Sch\"obel \cite{Sch93}. In 2016, Vu A. Le et al. \cite{VHTHT16} considered the class $\mathrm{Lie} \left(n, n-1\right)$ and 
	   classified its subclass containing all the algebras having 1-codimensional commutative derived ideal. One subclass in {\Li} 
	   was firstly considered and incompletely classified by Sch\"obel \cite{Sch93} in 1993. Later, Janisse also gave an incomplete 
	   classification of {\Li} and published as a scientific report \cite{Jan10} in 2010. In this paper, we set up a new approach to study the 
	   classifying problem of classes {\Li} as well as {\li} and present the new complete classification of {\Li} in the combination with 
	   the well-known Eberlein's result of 2-step nilpotent Lie algebras from \cite[p.\,37--72]{Ebe03}. The paper will also classify a subclass 
	   of {\li} and will point out missings in Sch\"obel \cite{Sch93}, Janisse \cite{Jan10}, Mubarakzyanov \cite{Mub63a} as well as revise an 
	   error of Morozov \cite{Mor58}.
	\end{abstract}

	\section{Introduction}\label{sec:1}

	   From historical point of view, Lie Theory was found by Marius Sophus Lie (1842--1899) in the last decades of the 19th century. 
	   Nowadays, one cannot deny that Lie Theory -- regarding Lie groups as well as Lie algebras -- is an important branch of mathematics 
	   which becomes more and more interesting because its applicable range has been expanded continuously not only in the inside of 
	   Mathematics but also in Modern Physics, Cosmology, Economics, Financial Mathematics, etc.
	   
	   As many areas of mathematics, one of the fundamental problems in Lie Theory is to classify all Lie algebras, up to an isomorphism. 
	   In particular, due to Levi and Maltsev's Theorems, the problem of classifying Lie algebras over a field of characteristic zero is reduced 
	   to the problem of classifying semi-simple and solvable ones, in which the semi-simple Lie algebras were completely solved by Cartan 
	   \cite{Car94} in 1894 (over complex field) and Gantmacher \cite{Gan39} in 1939 (over real field).
	   
	   Naturally, we have to classify solvable Lie algebras but it is much harder. Although several classifications in low dimension are known, 
	   the problem of the complete classification of the solvable Lie algebras (even if over the complex field) is still open. There are at least two 
	   ways of proceeding in the classification of solvable Lie algebras: \emph{by dimension} or \emph{by structure}.
	   
	   It seems to be impossible if we try to proceed with the classification by dimension, i.e. to classify Lie algebras with a fixed dimension, 
	   when the dimension is greater than 6 because the number of  parameters increases drastically and the volume of calculations, therefore, 
	   will become enormous. However, it is more effective to proceed by structure, i.e. to classify solvable Lie algebras with some specific 
	   given properties. In this paper, we follow the second way.
	   
	   Let us note that if $\G$ is an $n$-dimensional non-commutative solvable Lie algebra then its \emph{derived algebra}\footnote{Also 
	   known as the \emph{(first) derived ideal}, see Section \ref{sec:2} below.} $\G^1: = [\G, \G]$ has dimension $k \in \{1, 2, \ldots, n-1\}$. 
	   For convenience, we denote by $\mathrm{Lie} \left(n, k\right)$ the class of all $n$-dimensional real solvable Lie algebras having 
	   $k$-dimensional derived algebra ($k \in \{1, 2, \ldots, n-1\}$). In order to solve the difficult problem of classifying all solvable Lie algebras, it is natural to restrict this problem to the subclasses $\mathrm{Lie} \left(n, k\right)$, one by one, $1 \leqslant k \leqslant n-1$. First of all, we consider the classes $\mathrm{Lie} \left(n, k\right)$ or $\mathrm{Lie} \left(n, n-k\right)$ when $k$ is small. Sometimes, due to the complexity of the problem, we consider the subclass $\mathrm{Lie} \left(n, kC\right)$ of $\mathrm{Lie} \left(n, k\right)$ contains Lie algebras having ($k$-dimensional) commutative derived algebra.
	   
	   In an attempt to deal with the problem as above, some researchers studied the classes $\mathrm{Lie} \left(n, k\right)$ with 
	   $k \in \{1, 2, 3\}$ in recent decades. Namely, the complete classification of $\mathrm{Lie} \left(n, 1\right)$ was given by 
	   Sch\"obel \cite{Sch93} in 1993 which consists of the real affine Lie algebra or the real Heisenberg Lie algebras and their trivial 
	   extensions by commutative Lie algebras. Nevertheless, it is almost at the present unsolved if $k = 3$, while there are some results 
	   when $k=2$ as follows:
	   \begin{itemize}
	       \item Sch\"obel \cite{Sch93} in 1993 gave an incomplete classification of {\Li} based on the fact that if $\G$ belongs to 
	       {\Li} then it has a 4-dimensional subalgebra $S$ whose $[S, S]$ is 2-dimensional too. In Section \ref{sec:5}, we will point out 
	       that there is a missing in his classification.
	      \item Eberlein \cite{Ebe03} in 2003 gave a formal classification of 2-step nilpotent Lie algebras. That classification contains a 
	      special case of {\Li} when the considered Lie algebra has derived algebra which lies in its center.
	      \item Janisse \cite{Jan10} in 2010 considered the so-called \emph{structure matrix} whose elements are structure constants 
	      $a_{ij}^k$ with three indices: lexicographical order for pairs $(i, j)$ and normal one for $k$. In spite of interesting approach, 
	      we will show in Section \ref{sec:5} that his classification of {\Li} is also incomplete.
	   \end{itemize}
	   
	   Up to now, there is no more work which gives a complete classification of {\Li}. This motivates us to give in this paper a new 
	   approach to solve this problem. Roughly speaking, we will set up in this paper a new approach to study and classify the class {\Li}. 
	   Namely, we give a new complete classification of all non 2-step nilpotent algebras of {\Li}. Therefore, by combining a well-known 
	   Eberlein's classification of 2-step nilpotent Lie algebras in \cite[p.\,37--72]{Ebe03}, we obtain the new complete classification of {\Li}.
	   
	   More concretely, we use the well-known formula of the maximal dimension of commutative subalgebras contained in the Lie algebra 
	   $\M_n (\K)$ of $n$-square matrices with $\K$-valued entries. Schur \cite{Sch05} in 1905 is the first author who set up this formula 
	   over an algebraically closed field, and his result later was extended to an arbitrary field by Jacobson \cite{Jac44} in 1944. In fact, 
	   Proposition \ref{prop3} in Section \ref{sec:3} shows that if $\G$ belongs to {\Li} then its derived algebra must be commutative. 
	   As a consequence of this assertion, a suitable subalgebra of $\mathrm{Der} (\G) \equiv \M_n (\R)$ is commutative too. Fortunately, 
	   a mechanical combination of Schur and Jacobson's results with basic techniques of Linear Algebra as well as Lie Theory can 
	   give a complete classification of {\Li} as desired.
	   
	   Following the complete classification of $\mathrm{Lie} \left(n,(n-1)C\right)$ in \cite{VHTHT16}, we start with studying the class 
	   $\mathrm{Lie} \left(n,(n-2)C\right)$ ($n \geqslant 4$) in this paper by the similar method which is used to study the class {\Li}. 
	   Theorem \ref{thm2} in Section \ref{sec:3} gives an incomplete classification of $\mathrm{Lie} \left(n,(n-2)C\right)$. 
	   
	   The paper is organized into six sections, including this introduction. Some useful results and terminologies are listed in next section. 
	   Afterwards, Section \ref{sec:3} states the main results, and Section \ref{sec:4} is devoted to present the detailed proofs. 
	   Some comments based on comparison with previous works as well as some illustrations of the results in low dimensions are 
	   included in Section \ref{sec:5}. Finally, we present some concluding remarks in Section \ref{sec:6}.

	\section{Preliminaries}\label{sec:2}

	   In this section, we will recall some notions and well-known results which will be used later. First of all, we emphasize that, 
	   throughout this paper the notation $\M_n(\K)$ means the set of $n$-square matrices with entries in some field $\K$ and $\GL_n(\K)$ 
	   denotes the group of all invertible matrices in $\M_n(\K)$, where $n$ is a positive integer number. 

	   \begin{defn}
	   		Let $A, B$ be two $n$-square matrices in $\M_n(\K)$. We say that $A, B$ are \emph{proportional similar}, denoted by $A \sim_p B$, 
	   		if there exist $c \in \K \setminus \{0\}$ and $C \in \GL_n(\K)$ such that $cA = C^{-1}BC$.
	   \end{defn}

	   \begin{rem}
	      In fact, when $\K$ is the field of real numbers or complex numbers, the classification of $\M_n(\K)$ in proportional similar relation 
	      is easily reduced the well-known classification of $\M_n(\R)$ or $\M_n(\C)$ by using the standard Jordan form of square matrices.   
	   \end{rem}
	   
	   \begin{defn}
	      An \emph{$n$-dimensional Lie algebra over a field $\K$} is an $n$-dimensional vector space $\G$ over $\K$ together with 
	      a skew-symmetric bilinear map $[\cdot, \cdot]: \G \times \G \to \G$ which is called a \emph{Lie bracket} obeys Jacobi identity 
	      $[[X, Y], Z] + [[Z, X], Y] + [[Y, Z], X] = 0$ for all $X, Y, Z \in \G$. If $[\cdot, \cdot] \equiv 0$ then $\G$ is called \emph{commutative} or \emph{abelian}.
	   \end{defn}

	   \begin{defn}
	      A \emph{Lie isomorphism} $f: \G \to \mathcal{H}$ between two Lie algebras is a linear isomorphism which preserves Lie brackets, 
	      i.e. $f \left(\left[X, Y\right]_\G \right) = \left[f(X), f(Y)\right]_\mathcal{H}$ for all $X, Y \in \G$. If there exists a Lie isomorphism 
	      $f: \G \to \mathcal{H}$ then we say that $\G$ and $\mathcal{H}$ are \emph{isomorphic}.
	   \end{defn}
	   
	   \begin{rem}
	     For any Lie algebra $\G$ with a chosen basis $(X_1, \ldots, X_n)$, the Lie structure is absolutely defined by $[X_i, X_j],\,1 \leqslant i < j \leqslant n$. Sometimes, these Lie brackets are complex. Then, we will choose a suitable new basis such that the Lie brackets become simpler. 
	   \end{rem}
	   
	   \begin{defn}
	      A vector subspace $\mathcal{H}$ of a Lie algebra $\G$ is called a \emph{Lie subalgebra} of $\G$ if it is closed under the Lie bracket, 
	      i.e. $[X, Y] \in \mathcal{H}$ for all $X,Y \in \mathcal{H}$. Furthermore, a Lie subalgebra $\mathcal{H}$ is called an \emph{ideal} of 
	      $\G$ if $[X, Y] \in \mathcal{H}$ for all $X \in \G$ and $Y \in \mathcal{H}$. A Lie algebra $\G$ is said to be \emph{decomposable} if it is the 
	      direct sum of two non-trivial subalgebras, and \emph{indecomposable} otherwise.
	   \end{defn}
	   
	   \begin{defn}
	      Let $\G$ be a Lie algebra. We recall its three characteristic series as follows.
	      \begin{itemize}
	         \item The \emph{derived series $DS$} is
	         \[
	            \G^0 := \G \supset \G^1 := \left[ \G, \G \right] \supset \G^2 := \left[ \G^1, \G^1 \right] 
	            \supset \cdots \supset \G^k := \left[ \G^{k-1}, \G^{k-1} \right] \supset \cdots.
	         \]
	         We say that $\G$ is \emph{solvable} if $DS$ terminates, i.e. $\G^k = 0$ for some positive integer number $k$.
	         \item The \emph{lower central series $LS$} is	         
	         \[
	            \G_0 := \G \supset \G_1 := \left[ \G, \G \right] \supset \G_2 := \left[ \G, \G_1 \right] 
	            \supset \cdots \supset \G_k := \left[ \G, \G_{k-1} \right] \supset \cdots.
	         \]
	         Similarly, $\G$ is \emph{nilpotent} if $LS$ terminates, i.e. $\G_k = 0$ for some positive integer number $k$. Furthermore, 
	         if $\G_{k-1} \neq 0 = \G_k$ then we say that $\G$ is \emph{$k$-step nilpotent}.
	         \item The \emph{upper central series $US$} is
	         \[
	            0 := C_0 (\G) \subset C_1 (\G) = \Z(\G) \subset \cdots \subset C_k (\G) \subset \cdots,
	         \]
	         where $\Z(\G)$ is the center of $\G$, and $C_{k+1} (\G) := \{X \in \G: [X, \G] \subset C_k (\G)\}$. The sequence 
	         $\left(\dim C_1 (\G), \dim C_2(\G), \ldots, \dim C_k (\G), \ldots\right)$ is called the \emph{upper central series dimensions} of $\G$.
	      \end{itemize}
	   \end{defn}

	   \begin{rem}
	      It is obvious that $[\G, \G]$ is an ideal of $\G$ which is called its \emph{first derived ideal} or \emph{derived algebra} 
	      as we have emphasized in Section \ref{sec:1}.
	   \end{rem}

	   \begin{prop}[Schur-Jacobson Theorem \cite{Jac44,Sch05}]\label{Schur-JacobsonThm}
	      If $A$ is a commutative subalgebra of the Lie algebra $\M_n (\K)$ then $\dim A \leqslant \left[ \frac{n^2}{4} \right]+1$, 
	      where $[x]$ is the integer part of $x \in \R$.
	   \end{prop}

	   \begin{prop}[{{Lie's Theorem \cite[Theorem 1.25]{Kna02}}}]
	      Let $\K$ and $\Bbbk$ be subfields of the field of complex numbers $\C$ with $\Bbbk \subset \K$. Suppose that $\G$ is a solvable 
	      Lie algebra over $\Bbbk$ and $\rho: \G \to \E_\K V$ is a representation of $\G$ in a finite-dimensional vector space $V \neq 0$ over $\K$. 
	      If $\K$ is algebraically closed, then there is a simultaneous eigenvector $v \in V$ for all members of $\rho (\G)$. More generally, if all the 
	      eigenvalues of $\rho(X)$ lie in $\K$ for all $X \in \G$ then there is a simultaneous eigenvector too.
	   \end{prop}

	\section{Main results}\label{sec:3}
	
	   Throughout this paper, we will use the following notations:
	   \begin{enumerate}[$\bullet$]
	      \item In traditional notations, $\R$ (resp. $\C$) is the field of real (resp. complex) numbers.
	      \item Unless otherwise specified, $n$ will denote an integer number which is greater than 2.
	      \item $\s \{X_1, \ldots, X_n\}$ is the vector space spanned by the generating set $\{X_1, \ldots, X_n\}$.
	      \item The capital Gothic letter $\G$ indicates an $n$-dimensional real solvable Lie algebra, $\G^1 := [\G, \G]$ 
	      is the (first) derived ideal of $\G$ and $\mathrm{Der} \left(\G\right)$ is the Lie algebra of all derivations of $\G$.
	      \item $\mathrm{Lie} \left(n,k\right)$: the class of all real solvable Lie algebras having $k$-dimensional derived ideal.
	      \item $\mathrm{Lie} \left(n,kC\right) := \left\lbrace \G \in \mathrm{Lie} \left(n,k\right) \, \big| \, \G^1 := [\G, \G] \text{ is commutative} \right\rbrace$.
	      \item $a_X := \ad_{X}{\arrowvert}_{\G^1}$ is the restriction of the adjoint operator $\ad_X \in \mathrm{Der} (\G)$ on $\G^1$.
	      \item $\A_{\G}: = \s \{a_X : X \in \G\}$ is the Lie subalgebra of $\mathrm{Der} \left(\G^1\right)$ generated by $a_X$ for all $X \in \G$.
	      \item $\A_{\G}(\h)$ is the Lie subalgebra of $\G^1$ generated by $\bigcup_{X \in \G}\, a_X (\h)$, where $\h$ is a Lie subalgebra of $\G^1$.
	      \item $\af (\R)$: \emph{the real affine Lie algebra} ($2$-dimensional), i.e. $\af (\R) := \s \{X, Y\}$ with $[X, Y] = Y$.  
	      \item $\af(\C)$: \emph{the complex affine Lie algebra} ($4$-dimensional), i.e. $\af(\C) := \s \{X, Y, Z, T\}$ with non-trivial 
	      Lie brackets as follows $[Z,X] = -Y, [Z,Y] = X, [T,X]= X, [T,Y] = Y$.
	      \item $\h_{2m+1}$: \emph{the $(2m+1)$-dimensional real Heisenberg Lie algebra} ($m \geqslant 1$), i.e. 
	      \[
	         \h_{2m+1} := \s \left\lbrace X_i, Y_i, Z \; \big| \; i = 1, 2, \ldots, m \right\rbrace
	      \]
	      with non-trivial Lie brackets as follows $[X_i,Y_i] = Z, i = 1,2,\ldots,m$. 
	   \end{enumerate}

	   \subsection{\bf The new complete classification of non 2-step nilpotent Lie algebras of \Li}
	   
	   First of all, we give a sufficient and necessary condition to define a 2-step nilpotent Lie structure on $\G$ and give an upper bound 
	   of the dimension of $\A_{\G}$.

	   \begin{prop}\label{prop3}
	      Let $\G$ be an $n$-dimensional real solvable Lie algebra such that its derived ideal $\G^1$ is 2-dimensional. 
	      Then we have the following assertions
	      \begin{enumerate}
	         \item $\G^1$ must be commutative.
	         \item The Lie algebra $\A_{\G}$ is also commutative and $\dim \A_{\G} \leqslant 2$.
	         \item\label{prop3-3} $\G$ is 2-step nilpotent if and only if $\A_\G = 0$.
	      \end{enumerate}
	   \end{prop}
	   
	   \begin{proof}
	      \begin{enumerate}
	         \item We have known that if $\G^1$ is 2-dimensional then we always choose one basis $(X_1, X_2)$ such that 
	         $[X_1, X_2] = X_2$ or $[X_1, X_2] = 0$. Upon simple computation, by using the Jacobi identity for $(X_1, X_2, X)$ 
	         with $X$ is an arbitrary element from $\G$, we get that $[X_1, X_2] = X_2$ is impossible. That means $[X_1, X_2] = 0$, 
	         and $\G^1$ must be commutative.
	         \item Once again, applying the Jacobi identity for triple $(X, Y, Z)$ with all $X, Y \in \G$ and $Z \in \G^1$, we get that 
	         $a_X \circ a_Y = a_Y \circ a_X$. Therefore, $\A_\G$ is commutative. As a commutative subalgebra of the Lie algebra 
	         $\E \G^1 \cong \M_2(\R)$, it follows from Proposition \ref{Schur-JacobsonThm} that $\dim \A_\G \leqslant 2$.
	         \item It is easy to see that $\G_2: = \left[\G, \G^1\right] = \A_\G \left(\G^1\right)$. Therefore, $\G$ is 2-step nilpotent if and only if 
	         $\A_\G = 0$.
	      \end{enumerate}
	   \end{proof}

	   \begin{cor}\label{cor1}
	      Assume that $\G$ is a real solvable Lie algebra whose derived ideal $\G^1$ is 2-dimensional. Then $\G$ is not 2-step nilpotent 
	      if and only if $\dim \A_\G \in \{1, 2\}$.
	   \end{cor}

	   \begin{rem}\label{rem4}
	      As we have emphasized in Section \ref{sec:1}, Eberlein \cite{Ebe03} in 2003 studied the moduli space of 2-step nilpotent 
	      Lie algebras of type $(p, q)$ in which consisted of a classification of {\Li}. In particular, his result when $p = 2$ is a classification 
	      of {\Li} corresponding to $\A_\G = 0$. Thus, we only pay attention to the case $\dim \A_\G \in \{1, 2\}$.
	   \end{rem}
	   
	   Now we formulate the first main result of the paper in Theorem \ref{thm1} below. This theorem gives the (new) complete classification 
	   of all of the non 2-step nilpotent Lie algebras in \Li. In the list of algebras of the classification, $\G_{i,2.j(s)}$ means that the $j$-th 
	   Lie algebra of dimension $i$ whose derived algebra is 2-dimensional, and the last subscript, if any, is the  parameter on which the 
	   Lie algebra depends. For the sake of simplicity, we stipulate that, in the statements of results or the descriptions of Lie structure, 
	   we just list non-zero Lie brackets, i.e. all disappeared ones are trivial.

	   \begin{thm}[The complete classification of non 2-step nilpotent algebras in \Li]\label{thm1}
	      Let $\G$ be an $n$-dimensional real solvable Lie algebra such that its derived ideal $\G^1$ is 2-dimensional. 
	      We assume that $\G$ is not 2-step nilpotent. Then we can choose a suitable basis $(X_1, X_2, \ldots, X_n)$ of $\G$ such that 
	      $\G^1 = \s \{X_1, X_2\} \cong \R^2$ and the following assertions hold.
	      \begin{enumerate}
	         \item Assume that $\G$ is indecomposable.\label{part1-thm1}
	         \begin{enumerate}
	            \item[1.1] If $n = 3$ then the Lie structure of $\G$ is completely determined by the adjoint operator 
	            $a_{X_3} \in \Au \G^1 \equiv \GL_2 (\R)$ and $\G$ is isomorphic to one and only one of the Lie algebras as follows
	            \begin{enumerate}
	               \item[(i)] $\G_{3,2.1(\lambda)}:$ $a_{X_3} = \begin{bmatrix} 1 & 0 \\ 0 & \lambda \end{bmatrix}$ with $\lambda \in \R \setminus \{0\}$.
	               \item[(ii)] $\G_{3,2.2}:$ $a_{X_3} = \begin{bmatrix} 1 & 1 \\ 0 & 1 \end{bmatrix}$.
	               \item[(iii)] $\G_{3,2.3(\varphi)}:$ $a_{X_3} =  \begin{bmatrix} \cos \varphi & -\sin \varphi \\ \sin \varphi & \cos \varphi \end{bmatrix}$ 
	               with $\varphi \in (0, \pi)$.
	            \end{enumerate}
	            \item[1.2] If $n = 4$ then its Lie structure is defined by $[X_3, X_4]$ and two adjoint operators $a_{X_3}, a_{X_4} \in \E \G^1 
	            \equiv \M_2 (\R)$. Moreover, $\G$ is isomorphic to one and only one of the Lie algebras as follows
	            \begin{enumerate}
	               \item[(i)] $\G_{4,2.1}:$ $a_{X_3} = \begin{bmatrix} 1 & 0 \\ 0 & 0 \end{bmatrix}$ and $[X_3, X_4] = X_2$.
	               \item[(ii)] $\G_{4,2.2}:$ $a_{X_3} = \begin{bmatrix} 0 & 1 \\ 0 & 0 \end{bmatrix}$ and $[X_3, X_4] = X_2$.
	               \item[(iii)] $\G_{4,2.3(\lambda)}:$ $a_{X_3} = \begin{bmatrix} 0 & 1 \\ 0 & \lambda \end{bmatrix}$ and $a_{X_4} = 
	               \begin{bmatrix} 1 & 0 \\ 0 & 1 \end{bmatrix}, \, \lambda \in \R$.
	               \item[(iv)] $\G_{4,2.4} = \af (\C)$.	               
	            \end{enumerate}
	            \item[1.3] If $n = 5+2k \, (k \geqslant 0)$ then $\G \cong \G_{5+2k,2}:$ \, $a_{X_3} = \begin{bmatrix} 0 & 0 \\ 1 & 0 \end{bmatrix}$ and 
	            \[
	               [X_3, X_4] = X_1,\, [X_4, X_5] = \cdots = [X_{4+2k},\, X_{5+2k}] = X_2.
	            \]
	            \item[1.4] If $n = 6+2k \, (k \geqslant 0)$ then $\G$ is isomorphic to one and only one of the Lie algebras as follows
	            \begin{enumerate}
	               \item[(i)] $\G_{6+2k,2.1}:$ $a_{X_3} = \begin{bmatrix} 1 & 0 \\ 0 & 0 \end{bmatrix}$, 
	               $[X_3, X_4] = [X_5, X_6] = \cdots = [X_{5+2k}, X_{6+2k}] = X_2$. 
	               \item[(ii)] $\G_{6+2k,2.2}:$ $a_{X_3} = \begin{bmatrix} 0 & 0 \\ 1 & 0 \end{bmatrix}$, $[X_3, X_4] = X_1$, 
	               $[X_5, X_6] = \cdots = [X_{5+2k}, X_{6+2k}] = X_2$.
	            \end{enumerate}
	         \end{enumerate}
	         \item Assume that $\G$ is decomposable. Then we have 
	         \begin{enumerate}[2.1]
	            \item $\G \cong \af (\R) \oplus \af (\R)$ when $n = 4$ or $\G \cong \af (\R) \oplus \af (\R) \oplus \R^{n-4}$ when $n > 4$.
				\item $\G \cong \af (\R) \oplus \h_{2m+1}$ when $n = 2m+3$ or $\G \cong \af (\R) \oplus \h_{2m+1} \oplus \R^{n-2m-3}$ 
				when $n > 2m +3$, $m \geqslant 1$.
	            \item $\G$ is isomorphic to a trivial extension by a commutative Lie algebra of one of all the Lie algebras listed in Part \ref{part1-thm1}.
	         \end{enumerate}
	      \end{enumerate}
	   \end{thm}

	\subsection{\bf The classification of a subclass of $\mathrm{Lie} \left(n, (n-2)C\right)$}
	
	   In this section, we present the initial result in classification of a subclass of $\mathrm{Lie} \left(n, (n-2)C\right)$ which is, once again, 
	   an illustrative example of our new approach. Namely, we begin by considering the simplest case when 
	   $\G \in \mathrm{Lie} \left(n, (n-2)C\right)$, i.e. $\G^1 \cong \R^{n-2}$ is commutative. In other words, the second derived ideal $\G^2 = 0$. 
	   First of all, we also give in this case an upper bound of the dimension of $\A_\G$.

	    \begin{prop}\label{prop4}
	      Let $\G$ be an $n$-dimensional real solvable Lie algebra $(n \geqslant 4)$ such that its derived ideal ${\G}^1 \cong \R^{n-2}$. 
	      Then $\dim \A_\G \in \{1, 2\}$.
	   \end{prop}

	   \begin{proof}
	      Recall that $\A_\G := \s \{a_X: X \in \G\}$. Because $\G^1$ is commutative, $\A_\G = \s \{a_X: X\in \G \setminus \G^1\}$. 
	      Let $(X_1, X_2, \ldots, X_{n-2})$ be a basis of $\G^1$, by adding two linearly independent elements $Y, Z \in \G \setminus \G^1$, 
	      we get a basis $(X_1, X_2, \ldots, X_{n-2}, Y, Z)$ of $\G$. Then, $\A_\G = \s \{a_Y, a_Z\}$. From that we get $\dim \A_\G \leqslant 2$.
	      
	      However, if $\dim \A_\G = 0$, i.e. $a_Y = a_Z = 0$, then the Lie structure of $\G$ depends only on $[Y, Z]$. Furthermore, 
	      $\G^1= [\G, \G] = \s \{[Y, Z]\}$ implies that $n-2 = \dim \G^1 \leqslant 1$ which contradicts to $n \geqslant 4$. 
	      So $\dim \A_\G \in \{1, 2\}$.
	   \end{proof}
	   
	   In view of Proposition \ref{prop4} above, to completely classify $\mathrm{Lie} \left(n, (n-2)C\right)$, we have to consider two cases: 
	   $\dim \A_\G = 1$ or $\dim \A_\G = 2$. In this paper, we first consider the case $\dim \A_\G = 1$. In fact, Theorem \ref{thm2} below 
	   presents a classification of $\mathrm{Lie} \left(n, (n-2)C\right)$ when $\dim \A_\G = 1$. It is also the second main result of this paper.

	   \begin{thm}[The classification of $\mathrm{Lie} \left(n, (n-2)C\right)$ when $\dim \A_\G = 1$]\label{thm2}
	       Assume that $\G$ be an $n$-dimensional real solvable Lie algebra $(n \geqslant 4)$ whose derived ideal $\G^1 \cong \R^{n-2}$ 
	       and $\dim \A_\G = 1$. Then we can always choose two elements $Y, Z \in \G \setminus \G^1$ such that $Y, Z$ are linearly independent, 
	       $a_Y = 0 \neq a_Z$ and $\A_\G = \s \{a_Z \}$. Furthermore, we have the following assertions.
	       \begin{enumerate}
	          \item\label{part1-thm2} If $[Y, Z] = 0$ or $a_Z$ is non-singular then $\G$ is decomposable. Namely $\G \cong \R \oplus \bar{\G}$, 
	          where $\bar{\G} \in \mathrm{Lie} \left(n-1, n-2\right)$ and $\bar{\G}^1 = \left[\bar{\G},\bar{\G}\right] = \G^1$.
	          \item\label{part2-thm2} If $[Y, Z] \neq 0$ and $a_Z$ is singular then $\G$ is indecomposable. Moreover, the Lie structure of $\G$ is 
	          completely determined by $[Z, Y]$ and the operator $a_Z$. In this case, we can always choose a basis $(X_1, X_2, \ldots, X_{n-2})$ 
	          in $\G^1$ such that $[Z, Y] = X_{n-2}$, $\ra (a_Z) = n - 3$ and
	          \[
	             a_Z = \bar{A} \in \Bigg\{\begin{bmatrix} A & 0 \\ 0 & 0 \end{bmatrix}, \begin{bmatrix} 0 & A \\ 0 & 0 \end{bmatrix} 
	             \, \Big| \, A \in \GL_{n-3}(\R)\Bigg\}.
	          \]
	          In addition, two $(n-2)$-square real matrices $\bar{A}, \bar{B}$ define two isomorphic Lie structures on $\G$ if and only if 
	          $\bar{A} \sim_{p} \bar{B}$.
	       \end{enumerate}
	   \end{thm}
	   
	   \begin{rem}\label{rem5}
	      We have the following remarks.	
	      \begin{enumerate}
	         \item In fact, Part \ref{part2-thm2} of Theorem \ref{thm2} gives us a desired classification by using the well-known classification 
	         of real square matrices by the proportional similar relation.
	         \item For every $A, B \in \GL_{n-3}(\R)$, it is easy to see that
	         \[
	            \begin{array}{l l l}
	               \Bigg( \begin{bmatrix} A & 0 \\ 0 & 0\end{bmatrix} \sim_{p} \begin{bmatrix} B & 0 \\ 0 & 0\end{bmatrix} \Bigg) \Leftrightarrow 
	               \left( A \sim_{p} B \right), &
	               \Bigg( \begin{bmatrix} 0 & A \\ 0 & 0\end{bmatrix} \sim_{p} \begin{bmatrix} 0 & B \\ 0 & 0\end{bmatrix} \Bigg) \nRightarrow 
	               \left( A \sim_p B \right), &
	               \begin{bmatrix} A & 0 \\ 0 & 0\end{bmatrix} \nsim_p \begin{bmatrix} 0 & B \\ 0 & 0\end{bmatrix}.
	            \end{array}
	         \]
	      \end{enumerate}
	   \end{rem}

	\section{Proof of the main results}\label{sec:4}
	
	   \subsection{\bf Proof of Theorem \ref{thm1}}
	   
	   By Corollary \ref{cor1}, the proof of Theorem \ref{thm1} will be organized according to $\dim \A_{\G} = 1$ or $\dim \A_{\G} = 2$. 
	   In order to prove Theorem \ref{thm1} we need some lemmas below. Lemmas \ref{lem1}, \ref{lem2} and \ref{lem3} will discuss 
	   the case $\dim \A_\G = 1$, and the last one, Lemma \ref{lem4} will be an investigation when $\dim \A_\G = 2$.

	   \begin{lem}\label{lem1}
	      Let $\G$ be an $n$-dimensional real solvable Lie algebra such that its derived ideal $\G^1$ is 2-dimensional and $\A_\G$ is 1-dimensional. 
	      Then we can choose a suitable basis $(X_1, X_2, \ldots, X_n)$ of $\G$ such that $\G^1 = \s \{X_1, X_2\} \cong \R^2$, 
	      $\A_\G = \s \{a_{X_3}\}$ and $a_{X_i} = 0$ for all $i \geqslant 4$.
	   \end{lem}

	   \begin{proof}
	      Since $\G^1$ is commutative, we can choose $X_1, X_2$ such that $\G^1 = \s \{X_1, X_2\} \cong \R^2$. Moreover, we can choose 
	      $X_3 \in \G \setminus \G^1$ such that $\A_\G = \s \{a_{X_3}\}$. Then $X_1, X_2, X_3$ are linearly independent. 
	      We extend $(X_1, X_2, X_3)$ to a basis $(X_1,X_2, X_3, Y_4, \ldots, Y_n)$ of $\G$. Now there exist scalars 
	      $\alpha_i \, (i \geqslant 4)$ such that $a_{Y_i} = \alpha_ia_{X_3}$ because all $a_{Y_i} \in \A_\G$. By transformation
	      \[
	         X_i := Y_i - \alpha_iX_3, \quad i = 4, \ldots, n,
	      \]
	      we obtain a basis $(X_1,X_2, \ldots,X_n)$ of $\G$ such that $a_{X_i} = 0$ for all $i \geqslant 4$.
	   \end{proof}
	   
	   In order to fully describe the Lie structure of $\G$ it is necessary to determine not only the forms of $a_{X_3}$ but also the Lie brackets 
	   $[X_i,  X_j]$ and $[X_3, X_i] \; (i,j=4,\ldots, n)$. We now proceed by considering two cases of the operator $a_{X_3}$: 
	   non-singular case in Lemma \ref{lem2} and singular case in Lemma \ref{lem3}.

	   \begin{lem}\label{lem2}
	      Let $\G$ be an $n$-dimensional real solvable Lie algebra such that its derived ideal $\G^1$ is 2-dimensional, 
	      $\A_\G$ is 1-dimensional and $(X_1, X_2, \ldots, X_n)$ is the basis of $\G$ as in Lemma \ref{lem1}, i.e. 
	      $\G^1 = \s \{X_1, X_2\} \cong \R^2$ and $\A_\G = \s \{a_{X_3}\}$. Assume, in addition, that $a_{X_3}$ is non-singular. Then, we have
	      \begin{itemize}
	         \item If $n = 3$ then $\G$ is indecomposable, namely, $\G \cong \bar{\G}$, 
	         \item If $n > 3$ then $\G$ must be decomposable, namely, $\G \cong \bar{\G} \oplus \R^{n-3}$, 
	      \end{itemize}
	      where $\bar{\G}$ belongs to the set $\left\lbrace \G_{3,2.1(\lambda)}, \G_{3,2.2}, \G_{3,2.3(\varphi)} \, | \,\lambda \in 
	      \R \setminus \{0\}, \varphi \in (0, \pi) \right\rbrace$ of the algebras listed in the subcase 1.1 of Theorem \ref{thm1}.
	   \end{lem}

	   \begin{proof}
	      First of all, assume that $n > 3$. Then we set
	      \[
	         \begin{array}{l l l}
	            [X_3, X_k] := y_kX_1 + z_kX_2, && k \geqslant 4, \\
	            \left[X_i, X_j\right] := y_{ij}X_1 + z_{ij}X_2, && 4 \leqslant i < j \leqslant n.
	         \end{array}
	      \]
	      Now by using the Jacobi identity for triples $(X_3, X_i, X_j)$ we get
	      \begin{equation}\label{eq-lem2}
	         y_{ij}a_{X_3}(X_1) + z_{ij}a_{X_3}(X_2) = 0.
	      \end{equation}
	      By the non-singularity, $a_{X_3}$ is an isomorphism. In particular, $a_{X_3}(X_1)$ and $a_{X_3}(X_2)$ are linearly independent. 
	      Therefore, all $y_{ij} = z_{ij} = 0$. That means $[X_i,X_j] = 0$ for $4 \leqslant i < j \leqslant n$. Next, by setting
	      \[
	         X'_k := X_k - y'_{k}X_1 - z'_{k}X_2 \quad \text{where} \quad 
	         \begin{bmatrix} y'_k \\ z'_k \end{bmatrix} = a_{X_3}^{-1} \begin{bmatrix} y_k \\ z_k \end{bmatrix}, \, k = 4, \ldots, n,
	      \]
	      we get $[X_3, X'_k] = 0$ for all $k = 4, \ldots, n$. Therefore we can assume, without loss of generality, that 
	      $[X_3, X_k] = 0$ for all $k = 4, \ldots, n$. 
	      
	      All above arguments show that
	      \[
	         \G \cong 
	         \begin{cases}
	            \bar{\G} & \quad \text{if} \quad n = 3, \\
	            \bar{\G} \oplus \R^{n-3} & \quad \text{if} \quad n > 3,
	         \end{cases}
	      \]
	      where $\bar{\G} = \s \{X_1,X_2,X_3\}$ is a 3-dimensional real Lie algebra with $\bar{\G}^1 \equiv \G^1 = \s \{X_1, X_2\}$ 
	      and the Lie structure of $\bar{\G}$ is defined by the non-singular matrix $a_{X_3} = A \in \GL_2 (\R)$.
	      
	      Besides, it follows immediately from \cite[Theorem 4.5]{VHTHT16} that two Lie algebras $\bar{\G}$ defined respectively 
	      by $A, B \in \GL_2 (\R)$ are isomorphic if and only if $A \, \sim_{p} \, B$ (the proportional similar relation). In other words, 
	      there exists $c \in \R \setminus \{0\}$ such that the Jordan canonical forms of $cA$ and $B$ coincide. Therefore, the 
	      classification of $\bar{\G}$, in this case, is reduced to find out $\GL_2 (\R)/\sim_{p}$. Recall that the Jordan canonical 
	      classification of $\GL_2(\R)$ is given as follows
	      \[
	         \begin{bmatrix} \lambda_1 & 0 \\ 0 & \lambda_2 \end{bmatrix} (\lambda_1, \lambda_2 \in \R \setminus \{0\}); \quad
	         \begin{bmatrix} \lambda & 1 \\ 0 & \lambda \end{bmatrix} (\lambda \in \R \setminus \{0\}); \quad
	         \begin{bmatrix} a & -b \\ b & a \end{bmatrix} (a, b \in \R, b > 0).
	      \]
	      It is easily seen that 
	      \begin{itemize}
	         \item[$\bullet$]\label{case1-lem2} $\begin{bmatrix} \lambda_1 & 0 \\ 0 & \lambda_2 \end{bmatrix} \sim_{p} 
	         \begin{bmatrix} 1 & 0 \\ 0 & \lambda \end{bmatrix}$ where $\lambda = \frac{\lambda_2}{\lambda_1} \neq 0$.
	         \item [$\bullet$]\label{case2-lem2} $\begin{bmatrix} \lambda & 1 \\ 0 & \lambda \end{bmatrix} \sim_{p} 
	         \begin{bmatrix} 1 & 1 \\ 0 & 1 \end{bmatrix}$.
	         \item [$\bullet$]\label{case3-lem2} $\begin{bmatrix} a &  -b \\  b & a \end{bmatrix} \sim_{p} 
	         \begin{bmatrix} \cos \varphi &-\sin \varphi \\ \sin \varphi & \cos \varphi \end{bmatrix}$ where 
	         $\varphi = \arccos \frac{a}{\sqrt{a^2 + b^2}} \in (0, \pi)$.
	      \end{itemize}
	      Therefore, we get the classification of $\mathrm{GL}_2 (\R)$ by the proportional similar relation as follows
	      \[
	         \begin{bmatrix} 1 & 0 \\ 0 & \lambda \end{bmatrix} (\lambda \in \R \setminus \{0\}); \quad
	         \begin{bmatrix} 1 & 1 \\ 0 & 1 \end{bmatrix}; \quad
	         \begin{bmatrix} \cos \varphi &-\sin \varphi \\ \sin \varphi & \cos \varphi \end{bmatrix} (\varphi \in (0, \pi)).
	      \]
	      As an immediate consequence of \cite[Theorem 4.5]{VHTHT16}, we have 
	      \[
	         \bar{\G} \in \left\lbrace \G_{3,2.1(\lambda)}, \G_{3,2.2}, \G_{3,2.3(\varphi)} \, | \, \lambda \in \R \setminus \{0\}, \varphi \in (0, \pi)\right\rbrace.
	     \]
	     We emphasize that these algebras are exactly ones which are listed in {\bf subcase 1.1} of Theorem \ref{thm1}. This means that
	     \begin{itemize}
	        \item[$\bullet$] If $n = 3$ then $\G$ is indecomposable, namely, $\boldsymbol{\G \cong \bar{\G} \in \{ \G_{3,2.1(\lambda)}, \, \G_{3,2.2}, \, \G_{3,2.3(\varphi)}\}}$.
	        \item[$\bullet$] If $n > 3$ then $\G$ is decomposable, namely, $\boldsymbol{\G \cong \bar{\G} \oplus \R^{n-3}}$.
	     \end{itemize}
	     The proof of Lemma \ref{lem2} is complete.
	   \end{proof}

	   \begin{lem}\label{lem3}
	      Assume that $\G$ is an $n$-dimensional real solvable Lie algebra such that its derived ideal $\G^1$ is 2-dimensional, 
	      $\A_\G$ is 1-dimensional and $(X_1, X_2, \ldots, X_n)$ is the basis of $\G$ as in Lemma \ref{lem1}, i.e. 
	      $\G^1 = \s \{X_1, X_2\} \cong \R^2$ and $\A_\G = \s \{a_{X_3}\}$. Assume, in addition, that $a_{X_3}$ is singular. Then 
	      $n \geqslant 4$ and the following assertions hold.
	      \begin{enumerate}
	         \item If $n = 4$ then $\G$ is isomorphic to one and only one from the set $\{\G_{4,2.1}, \, \G_{4,2.2} \}$ of Lie algebras 
	         listed in subcase 1.2 of Theorem \ref{thm1}.
	         \item If $n > 4$ and $\G$ is indecomposable, then \,$\G$ is isomorphic to one and only one from the set 
	         $\{\G_{5 + 2k,2}\, | \, n = 2k + 5, \, k \geqslant 0\} \cup \{\G_{6 + 2k,2.1}, \, \G_{6 + 2k,2.2}\, | \, n = 2k + 6, \, k \geqslant 0\}$ of Lie algebras 
	         listed in subcases 1.3 and 1.4 of Theorem \ref{thm1}.
	         \item If $n > 4$ and $\G$ is decomposable, then \,$\G \cong \af (\R) \oplus \h_{2m+1} \, (n = 2m + 3, \, m \geqslant 1)$ or $\G$ 
	         is a trivial extension of $\mathcal{H}$ by a commutative Lie algebra, where 
	         \[
	            \mathcal{H} \in \{\af (\R) \oplus \h_{2m+1}, \, \G_{4,2.1}, \, \G_{4,2.2}, \, \G_{5 + 2k,2}, \, \G_{6 + 2k,2.1}, \, \G_{6 + 2k,2.2} \, | \, m \geqslant 1\}.
	         \]
	      \end{enumerate}
	   \end{lem}

	   \begin{proof}
	      In view of Lemma \ref{lem1}, $\G$ has a basis $(X_1, X_2, \ldots, X_n)$ such that $\G^1 = \s \{X_1,X_2\} \cong \R^2$ and 
	      $\A_\G = \s \{a_{X_3}\}$. It follows from the singularity that $\det(a_{X_3}) = 0$. Moreover, $n \geqslant 4$ because if it was not so, 
	      i.e. $n = 3$, then $\G^1 = \text{Im}(a_{X_3})$ must be 1-dimensional which conflicts with the assumption. It is clear that the characteristic 
	      polynomial of $a_{X_3} (\neq 0)$ is given as follows
	      \[
	         p(x) = x^2 - \T (a_{X_3}) x,
	      \]
	      where $\T$ is the trace of a matrix. In particular, $\lambda_1 = \T (a_{X_3})$ and $\lambda_2 = 0$ are eigenvalues of $a_{X_3}$. 
	      Therefore, we have the following mutually-exclusive possibilities.
	       \begin{itemize}
	          \item $\lambda_1 = \T (a_{X_3}) \neq 0 = \lambda_2$. In this case, $a_{X_3}$ is diagonalizable. We can choose a suitable basis 
	          $(X'_1, X'_2)$ of $\G^1$ to get $a_{X_3} = \begin{bmatrix} \lambda_1 & 0 \\ 0 & 0 \end{bmatrix}$. By setting 
	          $X'_3 = \frac{1}{\lambda_1}X_3$, we get $a_{X'_3} = \begin{bmatrix} 1 & 0 \\ 0 & 0 \end{bmatrix}$.
	          \item $\lambda_1 = \T (a_{X_3}) = 0 = \lambda_2$, i.e. zero is the unique eigenvalue of $a_{X_3}$. Since $a_{X_3} \neq 0$ we can 
	          convert it to the Jordan canonical form $\begin{bmatrix} 0 & 1 \\ 0 & 0 \end{bmatrix}$ in a suitable basis $(X'_1, X'_2)$ of $\G^1$.
	       \end{itemize}
	       In summary, without loss of generality, we can always assume that
	      \[
	         a_{X_3} = \begin{bmatrix} 1 & 0 \\ 0 & 0 \end{bmatrix} \quad \text{or} \quad a_{X_3} = \begin{bmatrix} 0 & 1 \\ 0 & 0 \end{bmatrix}.
	      \]
	      
	      Recall that we have set
	      \[
	         \begin{array}{l l l}
	            [X_3, X_k] := y_kX_1 + z_kX_2, && k \geqslant 4, \\
	            \left[X_i, X_j\right] := y_{ij}X_1 + z_{ij}X_2, && 4 \leqslant i<j \leqslant n.
	         \end{array}
	      \]
	      Besides, it follows from the equation \eqref{eq-lem2} that
	      \begin{itemize}
	         \item If $a_{X_3} = \begin{bmatrix} 1 & 0 \\ 0 & 0 \end{bmatrix}$ then $y_{ij} = 0$, i.e. $[X_i, X_j] = z_{ij}X_2$ for $4 \leqslant i < j \leqslant n$.
	         \item If $a_{X_3} = \begin{bmatrix} 0 & 1 \\ 0 & 0 \end{bmatrix}$ then $z_{ij}=0$, i.e. $[X_i, X_j] = y_{ij}X_1$ for $4 \leqslant i < j \leqslant n$.
	      \end{itemize}
	      Thus we have two mutually-exclusive possibilities as follows: $\G = \s \{X_1, X_2, \ldots, X_n\}$ whose Lie structure is one 
	      of two following systems
	      \[
	         \begin{array}{l l l l l}
	            \begin{array}{l r c r l}
	               [X_3, X_1] = & X_1, \\
	               \left[X_3, X_k\right] = & y_kX_1 & + & z_kX_2, & k \geqslant 4, \\
	               \left[X_i, X_j\right] = & & & z_{ij}X_2, & 4 \leqslant i < j \leqslant n, \\
	            \end{array} & & \text{or} & &
	            \begin{array}{l r c r l}
	               [X_3, X_2] = & X_1, \\
	               \left[X_3, X_k\right] = & y_kX_1 & + & z_kX_2, & k \geqslant 4, \\
	               \left[X_i, X_j\right] = & y_{ij}X_1, &&& 4 \leqslant i < j \leqslant n.
	            \end{array}
	         \end{array}
	      \]
	      
	      \begin{enumerate}[\bf 1]
	         \item {\bf The first case of Lemma \ref{lem3}}
	         
	         Now we consider the first case in which we have
	        \[
	         \begin{array}{l r c r l l}
	         [X_3, X_1] = & X_1, \\
	         \left[X_3, X_k\right] = & y_kX_1 & + & z_kX_2, && k \geqslant 4, \\
	         \left[X_i, X_j\right] = & & & z_{ij}X_2, && 4 \leqslant i < j \leqslant n. \\
	         \end{array}	         
	       \]
	       By using the following change of basis
	         \[
	            X'_i := X_i - y_iX_1, \quad i = 4, \ldots, n,
	         \]
	         we can reduce the Lie structure of $\G$ to
	         \[
	            \begin{array}{l r r l l}
	               [X_3, X_1] = & X_1, \\
	               \left[X_3, X_k\right] = & & z_kX_2, && k \geqslant 4, \\
	               \left[X_i, X_j\right] = & & z_{ij}X_2, && 4 \leqslant i < j \leqslant n.
	            \end{array}
	         \]
	         On the other hand, $\G^1 = \s \{X_1, X_2\}$ implies that all of $z_{ij}$ and $z_k$ are simultaneously non-vanished. 
	         Hence, there are exactly two mutually-exclusive subcases as follows.
	         
	         \begin{enumerate}[\bf A]
	            \item\label{case1A-lem3} {\bf The first subcase of Case 1 in Lemma \ref{lem3}: All $z_k = 0$, $k \geqslant 4$.} 
	            
	            Here, we have $n \geqslant 5$. In this subcase, it is easy to see that $\G$ must be decomposable. In fact 
	            $\G = \s \{X_1, X_3\} \oplus \s \{X_2, X_4, X_5, \ldots, X_n\}$. Moreover, we have that
	            \begin{itemize}
	               \item $\s \{X_1, X_3\}$ with $[X_3, X_1] = X_1$, i.e. $\s \{X_1, X_3\} \cong \af (\R)$.
	               \item $\s \{X_2, X_4, X_5, \ldots, X_n\}$ with $[X_i, X_j] = z_{ij}X_2 \, (4 \leqslant i < j \leqslant n)$ and there exists $z_{ij} \neq 0$. 
	               It is clear that this Lie subalgebra belongs to $\mathrm{Lie} \left(n, 1\right)$ which consists and only consists of 
	               the real affine Lie algebra $\af (\R)$ or the $(2m+1)$-dimensional real Heisenberg Lie algebra $\h_{2m+1}$ 
	               and their trivial extensions by a commutative Lie algebra (see \cite{VHTHT16}). Because $\s \{X_2, X_4, X_5, \ldots, X_n\}$ 
	               has at least one non-trivial Lie bracket of the form $[X_i, X_j] = z_{ij}X_2$, this Lie subalgebra must be the real Heisenberg 
	               Lie algebra or an its trivial extension.  
	            \end{itemize}
	            To summarize, we get the algebras listed in {\bf subcase 2.2} of Theorem \ref{thm1}. Namely, we have
	            \begin{itemize}
	               \item $\boldsymbol{\G \cong \af (\R) \oplus \mathfrak{h}_{2m+1}}$ \qquad \qquad \qquad \quad if \, $n = 2m + 3, \, m \geqslant 1$.
	               \item $\boldsymbol{\G \cong \af (\R) \oplus \mathfrak{h}_{2m+1} \oplus \R^{n-2m-3}}$ \quad if \, $n > 2m + 3, \, m \geqslant 1$.
	            \end{itemize}	             

	            \item\label{case1B-lem3} {\bf The second subcase of Case 1 in Lemma \ref{lem3}: There exists $z_k \neq 0$, $k \geqslant 4$.} 
	            
	            Here, we have $n \geqslant 4$. We will renumber $X_4, X_5, \ldots, X_n$, if necessary, to get $z_4 \neq 0$. 
	            Now, we use the following change of basis
	            \[
	               X'_4 := \frac{1}{z_4}X_4, \quad X'_k :=X_k - \frac{z_k}{z_4}X_4 \; (k \geqslant 5).
	            \]
	            That reduces the Lie structure of $\G$ to
	            \[
	               \begin{array}{l r r l l}
	                  [X_3, X_1] = & X_1, \\
	                  \left[X_3, X_4\right] = & & X_2, && \\
	                  \left[X_i, X_j\right] = & & z_{ij}X_2, && 4 \leqslant i < j \leqslant n.
	               \end{array}
	            \]
	            The next treatment procedure of the proof is as follows.
	            \begin{enumerate}
	               \item[(1B.1)] If all $z_{ij} = 0$ then $\G \cong \R^{n-4} \oplus \s \{X_1, X_2, X_3, X_4\}$ with $[X_3, X_1] = X_1$, 
	               $[X_3, X_4] = X_2$. This means $\boldsymbol{\G \cong \G_{4,2.1}}$ which is listed in {\bf subcase 1.2} of Theorem \ref{thm1}.
	               \item[(1B.2)] If there exists $z_{ij} \neq 0$ then $n \geqslant 5$. 
	               \begin{itemize}
	                  \item[$\bullet$] First, we consider $[X_4, X_k ] = z_{4k}X_2$ for $k \geqslant 5$. If there exists $z_{4k} \neq 0$ then by 
	                  renumbering $X_5, \ldots, X_n$, if necessary, we get $z_{45} \neq 0$. After that, we use the change of basis as follows
	                  \[
	                     X'_5 := \frac{1}{z_{45}}X_5,  \quad X'_k := X_k - \frac{z_{4k}}{z_{45}}X_5 \; (k \geqslant 6),
	                  \]
	                  then we have $[X_4, X_5] = X_2$ and $[X_4, X_k] = 0$ for all $k \geqslant 6$.
	                  \item[$\bullet$] Next, we consider $[X_5, X_k] = z_{5k}X_2$ for $k \geqslant 6$. If there exist $z_{5k} \neq 0$ then we 
	                  renumber $X_6, \ldots, X_n$ to get $z_{56} \neq 0$. By the same way as above, we will convert $z_{56}$ to 1 
	                  while $z_{5k} = 0$ for all $k \geqslant 7$.
	               \end{itemize}
	               We repeat this procedure as far as possible. It is obvious that this procedure must terminate because $\G$ is finite-dimensional. 
	               In other words, we can find out one integer $k$ such that
	               \begin{equation}\label{eq1-lem3}
	                  [X_3, X_4] = [X_4, X_5] = [X_5, X_6] = \cdots = [X_{3+k}, X_{4+k}] = X_2, \quad k \geqslant 1.
	               \end{equation}
	               Now we will refine the formula \eqref{eq1-lem3} according to the parity of $k$.
	               \begin{itemize}
	                  \item[$\bullet$] If $k =2l$ then \eqref{eq1-lem3} becomes to
	                  \[
	                     [X_3, X_4] = [X_4, X_5] = [X_5, X_6] = \cdots = [X_{3+2l}, X_{4+2l}] = X_2, \quad l \geqslant 1.
	                  \]
	                  Now we make $l$ changes of basis step by step as follows
	                  \[
	                     X'_{4+2i} = X_{4+2i} + X_{4+2i+2}, \quad  i = 0, \ldots, l-1.
	                  \]
	                  It is not hard to check that they convert \eqref{eq1-lem3} to the following formula
	                  \[
	                     [X_3, X_4] = [X_5, X_6] = \cdots = [X_{3+2l}, X_{4+2l}] = X_2, \quad l \geqslant 1.
	                  \]
	                   
	                  \item[$\bullet$] If $k = 2h+1$ then \eqref{eq1-lem3} becomes to
	                  \[
	                     [X_3, X_4] = [X_4, X_5] = [X_5, X_6] = \cdots = [X_{4+2h}, X_{5+2h}] = X_2, \quad h \geqslant 0.
	                  \]
	                  Now we make $h + 1$ changes of basis step by step as follows
	                  \[
	                     X'_{3+2i} = X_{3+2i} + X_{3+2i+2}, \quad  i = 0, \ldots, h.
	                  \]
	                  They also convert \eqref{eq1-lem3} to the following formula
	                  \[
	                     [X_4, X_5] = [X_6, X_7] = \cdots = [X_{4+2h}, X_{5+2h}] = X_2, \quad h \geqslant 0.
	                  \]
	               \end{itemize}
	               To summarize, we have the following Lie algebras
	               \begin{itemize}
	                  \item[$\bullet$] $\G \cong \R^{n-2k-6} \oplus \s \{X_1, X_2, \ldots, X_{6+2k}\}$ with $[X_3, X_1] = X_1$ and
	                  \[
	                     [X_3, X_4] = [X_5, X_6] = \cdots = [X_{5+2k}, X_{6+2k}] = X_2, \quad k \geqslant 0, n \geqslant 2k + 6.
	                  \]
	                  This means that $\boldsymbol{\G \cong \G_{6+2k,2.1}}$ if $n=2k+6, \, k \geqslant 0$ or $\G$ is isomorphic to 
	                  {\bf a trivial extension} (by a commutative Lie algebra) of $\boldsymbol{\G_{6+2k,2.1}}$ if $n > 2k+6, \, k \geqslant 0$, 
	                  where $\boldsymbol{\G_{6+2k,2.1}}$ is the first algebra listed in {\bf subcase 1.4} of Theorem \ref{thm1}.                 
	                  \item[$\bullet$] $\G \cong \R^{n-2k-5} \oplus \s \{X_1, X_2, \ldots, X_{5+2k}\}$ with $[X_3, X_1] = X_1$ and
	                  \[
	                     [X_4, X_5] = [X_6, X_7] = \cdots = [X_{4+2k}, X_{5+2k}] = X_2, \quad k \geqslant 0, n \geqslant 2k + 5,
	                  \]
	                  This means that
	                  \[
	                     \G \cong \s \{X_1, X_3\} \oplus \s \{X_2, X_4, \ldots, X_n\},
	                  \]
	                  where $\s\{X_1, X_3\} \cong \af (\R)$ and $\s\{X_2, X_4, \ldots, X_n\}$ is the Heisenberg Lie algebras or an its trivial extension 
	                  by a commutative Lie algebra. Therefore, $\boldsymbol{\G \cong \mathcal{H}}$ if $n = 2m+3 = 2k+5, m = k+1 \geqslant 1$ 
	                  or $\G$ is isomorphic to {\bf a trivial extension} (by a commutative Lie algebra) of $\boldsymbol{\mathcal{H}}$, if $n > 2m +3$, 
	                  where $\boldsymbol{\mathcal{H} = \af (\R) \oplus \h_{2m+1}}$\, ($m \geqslant 1$) is the first algebra listed in {\bf subcase 2.2} 
	                  of Theorem \ref{thm1}.
	                \end{itemize}
	            \end{enumerate}
	         \end{enumerate}

	         \item\label{case2-lem3} {\bf The second case of Lemma \ref{lem3}}
	         
	         Now we consider the second case in which we have
	         \[
	            \begin{array}{l r c r l l}
	               [X_3, X_2] = & X_1, \\
	               \left[X_3, X_k\right] = & y_kX_1 & + & z_kX_2, && k \geqslant 4, \\
	               \left[X_i, X_j\right] = & y_{ij}X_1, &&&& 4 \leqslant i < j \leqslant n.
	            \end{array}
	         \]	         
	         First of all, we use the change of basis as follows
	         \[
	            X'_1 := X_2, \quad X'_2 := X_1, \quad X'_k := X_k - y_kX_2 \, (k \geqslant 4).
	         \]
	        This reduces the Lie structure of $\G$ to
	         \[
	            \begin{array}{l r r l l l}
	               [X_3, X_1] = & & X_2, \\
	               \left[X_3, X_k\right] = & z_kX_1, &&&& k \geqslant 4, \\
	               \left[X_i, X_j\right] = && y_{ij}X_2, &&& 4 \leqslant i < j \leqslant n.
	            \end{array}
	         \]
	         On the other hand, because $\G^1 = \s \{X_1, X_2\}$ there exists $z_k \neq 0$. Without loss of generality, we may assume 
	         that $z_4 \neq 0$. Then by transformation
	         \[
	            X'_4 := \frac{1}{z_4}X_4, \quad X'_k := X_k - \frac{z_k}{z_4}X_4 \; (k \geqslant 5),
	         \]
	         it reduces the Lie structure of $\G$ to
	         \[
	            \begin{array}{l r r l l l}
	               [X_3, X_1] = & & X_2, \\
	               \left[X_3, X_4 \right] = & X_1, \\
	               \left[X_i, X_j \right] = & & y_{ij}X_2, && 4 \leqslant i < j \leqslant n.
	            \end{array}
	         \]
	         Now the next treatment procedure to deal with last Lie brackets is absolutely similar to subcase \ref{case1B-lem3}. 
	         More precisely, we also have two following subcases.
	         \begin{enumerate}[\bf A]

	         	\item\label{case2A-lem3} {\bf The first subcase of Case 2 in Lemma \ref{lem3}: All $y_{ij} = 0$, $4 \leqslant i < j \leqslant n$.}
	         	
	         	Here we obtain
	         	\[
	         	   \G \cong \R^{n-4} \oplus \s \{X_1, X_2, X_3, X_4\}, [X_3, X_1] = X_2, [X_3, X_4] = X_1, n \geqslant 4.
	         	\]
	         	By changing the role of $X_1$ and $X_2$ we convert this subcase to the one as follows 
	         	\[
	         	   \G \cong \R^{n-4} \oplus \s \{X_1, X_2, X_3, X_4\}, [X_3, X_2] = X_1, [X_3, X_4] = X_2, n \geqslant 4.
	         	\]
	         	In fact, we get
	         	\begin{itemize}
	         	   \item[$\bullet$] $\boldsymbol{\G \cong \G_{4,2.2}}$ \qquad \qquad \quad  if \, $n = 4,$
	         	   \item[$\bullet$] $\boldsymbol{\G \cong \G_{4,2.2} \oplus \R^{n-4}}$  \,\quad if \, $n > 4$,
	         	\end{itemize}
	         	where $\boldsymbol{\G_{4,2.2}}$ is the second algebra listed in {\bf the subcase 1.2} of Theorem \ref{thm1}. We also emphasize 
	         	that in this subcase, $\G$ is 3-step nilpotent.

	         	\item\label{case2B-lem3} {\bf The second subcase of Case 2 in Lemma \ref{lem3}: There exists $y_{ij} \neq 0$, 
	         	$4 \leqslant i < j \leqslant n$.}
	         	
	         	Here $n \geqslant 5$ and the last Lie brackets become to
	         	\begin{equation}\label{eq2-lem3} 
	         	   [X_4, X_5] = [X_5, X_6] = \cdots = [X_{4+k}, X_{5+k}] = X_2, \quad k \geqslant 0.
	         	\end{equation}
	         	We also refine \eqref{eq2-lem3} according to the parity of $k$ to obtain two Lie algebras as follows
	         	\begin{itemize}
	         	   \item[$\bullet$] $\G \cong \R^{n-2k-5} \oplus \s \{X_1, X_2, \ldots, X_{5+2k}\}$ with $[X_3, X_4] = X_1$ and
	         	   \[
	         	      [X_3, X_1] = [X_4, X_5] = \cdots = [X_{4+2k}, X_{5+2k}] = X_2, \quad k \geqslant 0, n \geqslant 2k + 5.
	         	   \]
	         	   In fact, we get here $\boldsymbol{\G \cong \G_{5+2k,2}}$ when $n=2k+5 \, (k \geqslant 0)$ or $\G$ is isomorphic to one 
	         	   {\bf its trivial extension} (by a commutative Lie algebra) when $n > 2k + 5 \, (k \geqslant 0)$, where $\boldsymbol{\G_{5+2k,2}}$ 
	         	   is the algebra listed in {\bf the subcase 1.3} of Theorem \ref{thm1}.    	
	         	   \item[$\bullet$] $\G \cong \R^{n-2k-6} \oplus \s \{X_1, X_2, \ldots, X_{6+2k}\}$ with $[X_3, X_4] = X_1$ and
	         	   \[
	         	      [X_3, X_1] = [X_5, X_6] = \cdots = [X_{5+2k}, X_{6+2k}] = X_2, \quad k \geqslant 0, n \geqslant 2k + 6.
	         	   \]
	         	   Here we get $\boldsymbol{\G \cong \G_{6+2k,2.2}}$ if $n=2k+6, \, k \geqslant 0$, or $\G$ is isomorphic to 
	         	   {\bf a trivial extension} (by a commutative Lie algebra) of $\boldsymbol{\G_{6+2k,2.2}}$ if $n > 2k+6, \, k \geqslant 0$. 
	         	   Note that $\boldsymbol{\G_{6+2k,2.2}}$ is the second algebra listed in {\bf the subcase 1.4} of Theorem \ref{thm1}.	         	   
	         	\end{itemize}
	         \end{enumerate}
	         It is easy to check that all the Lie algebras in this subcase are 3-step nilpotent.
	      \end{enumerate}
	       The proof of Lemma \ref{lem3} is complete.
	   \end{proof}

	   \begin{rem}
	      So far, we have concerned ourselves only with the case $\dim \A_{\G} = 1$. By the results just obtained from Lemmas 
	      \ref{lem1}, \ref{lem2}, \ref{lem3}, most of Lie algebras listed in Theorem \ref{thm1} have appeared in the results, except for 
	      $\G_{4,2.3(\lambda)} \,(\lambda \in \R)$, $\G_{4,2.4} = \af (\C)$ in subcase 1.2, $\af (\R) \oplus \af (\R)$ in subcase 2.1 and their 
	      trivial extensions. Now we turn to the case $\dim \A_{\G} = 2$.  
	   \end{rem}

	   \begin{lem}\label{lem4}
	      Let $\G$ be an $n$-dimensional real solvable Lie algebra such that its derived ideal $\G^1$ is 2-dimensional and $\A_\G$ 
	      is 2-dimensional. Then, we can choose a suitable basis $(X_1, X_2, \ldots, X_n)$ of $\G$ such that $\G^1 = \s \{X_1, X_2\} \cong \R^2$, 
	      $\A_{\G} = \s \{a_{X_3},a_{X_4}\}$ and the following assertions hold.
	      \begin{enumerate}
	         \item $\G \cong \af (\R) \oplus \af (\R)$ or $\G \cong \af (\R) \oplus \af (\R) \oplus \R^{n-4} \, (n>4)$.
	         \item $\G \cong \G_{4,2.3(\lambda)}$ or $\G \cong \G_{4,2.3(\lambda)} \oplus  \R^{n-4} \, (n>4)$, where 
	         $\G_{4,2.3(\lambda)} \, (\lambda \in \R)$ is the Lie algebras listed in subcase 1.2 of Theorem \ref{thm1}. 
	         \item $\G \cong \af (\C)$ or $\G \cong \af (\C) \oplus \R^{n-4}$. 
	      \end{enumerate}
	   \end{lem}

	   \begin{proof}
	      Firstly, we choose a basis $(X_1, X_2)$ of $\G^1 \cong \R^2$. By assumption $\dim \A_\G = 2$, we can always choose two 
	      distinct elements $X_3, X_4 \in \G \setminus \G^1$ such that $(a_{X_3}, a_{X_4})$ is a basis of $\A_{\G}$. 
	      By adding $Y_5, \ldots, Y_n$, if necessary (i.e. when $n > 4$), we get the basis $(X_1, X_2, X_3, X_4, Y_5, \ldots, Y_n)$ of $\G$.
	      
	      Because $a_{Y_i} \in \A_\G = \s \{a_{X_3}, a_{X_4}\}$\,($i = 5, \ldots, n$), we can assume that
	      \[
	         a_{Y_i} = \alpha_ia_{X_3} + \beta_ia_{X_4}, \quad i = 5, \ldots, n.
	      \]
	      By setting $X_i = Y_i - \alpha_iX_3 - \beta_iX_4$ we get $a_{X_i} = 0$ for all $i \geqslant 5$. Therefore, we get a new basis 
	      $(X_1, X_2, \ldots, X_n)$ of $\G$ such that $\G^1 = \s \{X_1, X_2\} \cong \R^2$, $\A_\G = \s \{a_{X_3}, a_{X_4}\}$ and $a_{X_i} = 0$ 
	      for all $i \geqslant 5$. Next, we will proceed the proof of the lemma by considering three mutually-exclusive cases of $a_{X_3}$ 
	      and $a_{X_4}$ as follows.

	      \begin{enumerate}[\bf 1.]
	         \item\label{case1-lem4} {\bf The first case of Lemma \ref{lem4}: $a_{X_3}$ and $a_{X_4}$ have real eigenvalues}
	         
	         In this case, according to Lie's Theorem, $a_{X_3}$ and $a_{X_4}$ must have at least one common eigenvector $ T_0 \in \G^1$. 
	         We can assume that
	         \[
	            a_{X_3}(T_0) = \lambda_{X_3}T_0, \quad a_{X_4}(T_0) = \lambda_{X_4}T_0, \quad (\lambda_{X_3}, \lambda_{X_4} \in \R).
	         \]
	         Then by setting $\Lambda(a_{X_3}) := \lambda_{X_3}$ and $\Lambda(a_{X_4}) := \lambda_{X_4}$ and linear extending we 
	         get a weight function $\Lambda: \A_\G \to \R$. We now consider the weight space of $\G^1$ corresponding to weight function 
	         $\Lambda$ as follows
	         \[
	            E_{\Lambda} := \{T \in \G^1: a_X(T) = \Lambda(a_X)T \text{ for all } X \in \G\}.
	         \]
	         According to Lie's Theorem, $E_{\Lambda}$ is a non-trivial subspace of $\G^1$. Therefore, we have two mutually-exclusive 
	         subcases as follows.
	         
	         \begin{enumerate}[\bf A]
	            \item\label{case1A-lem4} {\bf The first subcase of Case 1 in Lemma \ref{lem4}: $\dim E_{\Lambda} = 2$.}
	            
	            Because $a_{X_3}$ and $a_{X_4}$, in this subcase, have two linearly independent eigenvectors, they are diagonalizable. 
	            So there exists a basis of $E_{\Lambda} \equiv \G^1$ such that both $a_{X_3}$ and $a_{X_4}$ have the following diagonal forms
	            \begin{equation}\label{eq1-lem4}
	               a_{X_3} = \begin{bmatrix} a & 0 \\ 0 & b \end{bmatrix}, \quad
	               a_{X_4} = \begin{bmatrix} c & 0 \\ 0 & d \end{bmatrix}, \quad \left(a^2 + b^2 \neq 0 \neq c^2 + d^2\right).
	            \end{equation}
	            Without loss of generality, we can assume that the basis $(X_1, X_2, \ldots, X_n)$ satisfies diagonal condition of $a_{X_3}$ 
	            and $a_{X_4}$ as in \eqref{eq1-lem4}. Moreover, matrix $\begin{bmatrix} a & b \\ c & d \end{bmatrix}$ is invertible due to 
	            linear independence of $a_{X_3}$ and $a_{X_4}$. We set
	            \[
	               \begin{bmatrix} x_{11} & x_{12} \\ x_{21}& x_{22} \end{bmatrix} = \begin{bmatrix} a & b \\ c & d \end{bmatrix}^{-1}.
	            \]
	            That means
	            \begin{equation}\label{eq2-lem4}
	               \begin{cases}
	                  ax_{11} + cx_{12} = 1,\\
	                  bx_{11} + dx_{12} = 0,\\
	               \end{cases} \quad \text{and} \quad
	               \begin{cases}
	                  ax_{21} + cx_{22} = 0,\\
	                  bx_{21} + dx_{22} = 1.\\
	               \end{cases}
	            \end{equation}
	            Now, by using the transformation
	            \[
	               X'_3 = x_{11}X_3 + x_{12}X_4, \quad X'_4 = x_{21}X_3 + x_{22}X_4,
	            \]
	            and taking account of \eqref{eq2-lem4}, we convert $a_{X_3}$ and $a_{X_4}$ to the following forms
	            \[
	               a_{X_3} = \begin{bmatrix} 1 & 0 \\ 0 & 0 \end{bmatrix} \quad \text{and} \quad
	               a_{X_4} = \begin{bmatrix} 0 & 0 \\ 0 & 1 \end{bmatrix}.
	            \]
	            Now we deal with the remaining Lie brackets. For convenience, we set 
	            \[
	               \begin{array}{l l l}
	                  [X_3, X_4] = \alpha X_1 + \beta X_2, \\
	                  \left[X_3, X_k\right] = a_kX_1 + b_kX_2, && k \geqslant 5, \\
	                  \left[X_4, X_k\right] = c_kX_1 + d_kX_2, && k \geqslant 5, \\
	                  \left[X_i, X_j\right] = y_{ij}X_1 + z_{ij}X_2, && 5 \leqslant i < j \leqslant n.
	               \end{array}
	            \]
	            \begin{itemize}
	               \item[$\bullet$] Using the Jacobi identity for triples $(X_3, X_i, X_j)$ and $(X_4, X_i, X_j)$ we get $y_{ij} = z_{ij} = 0$ for all 
	               $5 \leqslant i < j \leqslant n$.
	               \item[$\bullet$] Similarly, using the Jacobi identity for triples $(X_3, X_4, X_k)$ we get all $b_k = c_k = 0$. Then by the 
	               transformation $X'_k := X_k - a_kX_1 - d_kX_2$ we reduce all $a_k$ and $d_k$ to zero value.
	               \item[$\bullet$] Finally, by the change of basis $X'_3 = X_3 + \beta X_2$ and $X'_4 = X_4 - \alpha X_1$ we also reduce the 
	               values of $\alpha$ and $\beta$ to zero.
	            \end{itemize}
	            To summarize, in this subcase, $\G \cong \s \{X_1, X_2, X_3, X_4\} \oplus \R^{n-4}$ when $n \geqslant 4$, with
	            \[
	               a_{X_3} = \begin{bmatrix} 1 & 0 \\ 0 & 0 \end{bmatrix} \quad \text{and} \quad
	               a_{X_4} = \begin{bmatrix} 0 & 0 \\ 0 & 1 \end{bmatrix}.
	            \]
	            This means
	            \begin{itemize}
	               \item[$\bullet$]$\boldsymbol{\G \cong \af (\R) \oplus \af (\R)}$ \qquad \qquad \quad when $n=4$.
	               \item[$\bullet$]$\boldsymbol{\G \cong \af (\R) \oplus \af (\R) \oplus \R^{n-4}}$ \quad when $n > 4$.
	            \end{itemize}
	            Note that, these algebras are exactly ones listed in {\bf subcase 2.1} of Theorem \ref{thm1}.

	            \item\label{case1B-lem4} {\bf The second subcase of Case 1 in Lemma \ref{lem4}: $\dim E_\Lambda = 1$.}
	            	            
	            In this subcase, we can always assume, without loss of generality, that $X_1$ is a simultaneous eigenvector of $a_{X_3}$ 
	            and $a_{X_4}$. Hence, both $a_{X_3}$ and $a_{X_4}$ have the form of the upper triangle matrices.
	            
	             First of all, using the well-known techniques in Linear algebra and noting that the role of $X_3, X_4$ is not equal, we can 
	             always reduce $a_{X_4}$ to one and only one of the following forms
	            \[
	               \begin{bmatrix} 0 & 0 \\ 0 & 1 \end{bmatrix}, \quad
	               \begin{bmatrix} 1 & 0 \\ 0 & \mu \end{bmatrix}, \quad
	               \begin{bmatrix} \mu & 1 \\ 0 & \mu \end{bmatrix} \quad (\mu \in \R),
	            \]
	            while $a_{X_3}$ get the form as follows	           
	               $a_{X_3} = \begin{bmatrix} a_{11} & a_{12} \\ 0 & a_{22} \end{bmatrix} \, \neq 0$.

	            Now we have three mutually-exclusive subcases due to the above types of $a_{X_4}$.
	            \begin{itemize}
	               \item[$\bullet$] Firstly, assume that $a_{X_4} = \begin{bmatrix} 0 & 0 \\ 0 & 1 \end{bmatrix}$. Then $X_2$ is also a eigenvector 
	               of $a_{X_4}$. Therefore, $X_2$ is not a eigenvector of $a_{X_3}$ since $\dim E_{\Lambda} = 1$, i.e. $a_{12} \neq 0$. 
	               Moreover, it follows from the commutativity of $a_{X_3}$ with $a_{X_4}$ that $a_{12} = 0$. This contradiction shows that this 
	               subcase is impossible.
	               
	               \item[$\bullet$] Next, assume that $a_{X_4} = \begin{bmatrix} 1 & 0 \\ 0 & \mu \end{bmatrix}$\,($\mu \in \R$). Then $X_2$ is 
	               also a eigenvector of $a_{X_4}$. In the same way as above, we get $a_{12} \neq 0$. Here, the commutativity of $a_{X_3}$ 
	               with $a_{X_4}$ shows that
	               \[
	                  a_{12}(\mu - 1) = 0 \quad \Longrightarrow \quad \mu = 1 \quad \Longrightarrow \quad 
	                  a_{X_4} = \begin{bmatrix} 1 & 0 \\ 0 & 1 \end{bmatrix}.
	               \]
	               Now, by setting $X'_3 = \frac{1}{a_{12}}X_3 - \frac{a_{11}}{a_{12}}X_4$ we convert $a_{X_3}$ to the form 
	               $\begin{bmatrix} 0 & 1 \\ 0 & \lambda \end{bmatrix}$, where $\lambda = \dfrac{a_{22} - a_{11}}{a_{12}} \in \R$. 
	               In other words, we can always assume, in this subcase, that
	               \[
	                  a_{X_3} = \begin{bmatrix} 0 & 1 \\ 0 & \lambda \end{bmatrix} \; (\lambda \in \R) \quad \text{and} \quad
	                  a_{X_4} = \begin{bmatrix} 1 & 0 \\ 0 & 1 \end{bmatrix}.
	               \]
               
	               \item[$\bullet$] Lastly, assume that $a_{X_4} = \begin{bmatrix} \mu & 1 \\ 0 & \mu \end{bmatrix} \, (\mu \in \R)$. Then, the 
	               commutativity of $a_{X_3}$ with $a_{X_4}$ implies $a_{11} = a_{22}$. By setting $a = a_{11} = a_{22}$ and $b = a_{12}$ 
	               we get $a_{X_3} = \begin{bmatrix} a & b \\ 0 & a \end{bmatrix}$.
	               \begin{itemize}
	                  \item If $b = 0$ then $a_{X_3} = \begin{bmatrix} a & 0 \\ 0 & a \end{bmatrix} \neq 0$. In particular, $a \neq 0$. 
	                  Now, we use the change of basis as follows
	                  \[
	                     X'_3 := -\frac{\mu}{a}X_3 + X_4, \quad X'_4 := \frac{1}{a}X_3.
	                  \]
	                 Then, we reduce $a_{X_3}$ and $a_{X_4}$ to the following forms
	                  \[
	                     a_{X_3} = \begin{bmatrix} 0 & 1 \\ 0 & 0 \end{bmatrix} \quad \text{and} \quad
	                     a_{X_4} = \begin{bmatrix} 1 & 0 \\ 0 & 1 \end{bmatrix}.
	                  \]
	                  Here, we emphasize that $a_{X_3} = \begin{bmatrix} 0 & 1 \\ 0 & \lambda \end{bmatrix}$ with $\lambda = 0$.	                  
	                  \item  If $b \neq 0$, replacing $X_3$ by $\frac{1}{b}X_3$ we can convert $a_{X_3}$ to the form 
	                  $\begin{bmatrix} \nu & 1 \\ 0 & \nu \end{bmatrix}$, where $\nu = \frac{a}{b} \in \R$. Since $a_{X_3}$ and $a_{X_4}$ are 
	                  linearly independent, $\mu \neq \nu$. Now we change basis as follows
	                  \[
	                     X'_3 = \frac{\mu}{\mu - \nu}X_3 - \frac{\nu}{\mu - \nu}X_4, \quad
	                     X'_4 = -\frac{1}{\mu - \nu}X_3 + \frac{1}{\mu - \nu}X_4.
	                  \]
	                 Once again, we reduce $a_{X_3}$ and $a_{X_4}$ to the following forms
	                  \[
	                     a_{X_3} = \begin{bmatrix} 0 & 1 \\ 0 & 0 \end{bmatrix} \quad \text{and} \quad
	                     a_{X_4} = \begin{bmatrix} 1 & 0 \\ 0 & 1 \end{bmatrix},
	                  \]
	                  i.e. $a_{X_3} = \begin{bmatrix} 0 & 1 \\ 0 & \lambda \end{bmatrix}$ with $\lambda = 0$.	                  
	               \end{itemize}
	            \end{itemize}
	            
	            Combining all above arguments, we can always assume that
	            \[
	               a_{X_3} = \begin{bmatrix} 0 & 1 \\ 0 & \lambda \end{bmatrix} \; (\lambda \in \R) \quad \text{and} \quad
	               a_{X_4} = \begin{bmatrix} 1 & 0 \\ 0 & 1 \end{bmatrix}.
	            \]
	            Now, we deal with the remaining Lie brackets. For convenience, we set
	            \[
	               \begin{array}{l l l}
	                  [X_3, X_4] = \alpha X_1 + \beta X_2, \\
	                  \left[X_3, X_k\right] = a_kX_1 + b_kX_2, && k \geqslant 5, \\
	                  \left[X_4, X_k\right] = c_kX_1 + d_kX_2, && k \geqslant 5, \\
	                  \left[X_i, X_j\right] = y_{ij}X_1 + z_{ij}X_2, && 5 \leqslant i < j \leqslant n.
	               \end{array}
	            \]
	            \begin{itemize}
	               \item[$\bullet$] Firstly, using the Jacobi identity for triples $(X_4, X_i, X_j)$ we get all $y_{ij} = z_{ij} = 0$.
	               \item[$\bullet$] Next, by the transformation $X'_k := X_k - c_kX_1 - d_kX_2$ we convert the values of  $c_k$ and $d_k$ to zero. 
	               \item[$\bullet$] After that, using the Jacobi identity for triples $(X_3, X_4, X_k)$ we get $a_k = b_k = 0$.
	               \item[$\bullet$] Finally, we also reduce the values of $\alpha$ and $\beta$ to zero by using the transformation 
	               $X'_3 = X_3 + \alpha X_1 + \beta X_2$.
	            \end{itemize}
	            To summarize, in this subcase we have $\G \cong  \s \{X_1, X_2, X_3, X_4\} \oplus \R^{n-4}$ with
	            \[
	               a_{X_3} = \begin{bmatrix} 0 & 1 \\ 0 & \lambda \end{bmatrix} \; (\lambda \in \R) \quad \text{and} \quad
	               a_{X_4} = \begin{bmatrix} 1 & 0 \\ 0 & 1 \end{bmatrix}.
	            \]
	         \end{enumerate}
	         This means
	         \begin{itemize}
	            \item[$\bullet$]$\boldsymbol{\G \cong \G_{4,2.3(\lambda)}}$ \qquad \qquad \quad when $n=4$,
	            \item[$\bullet$]$\boldsymbol{\G \cong \G_{4,2.3(\lambda)}\oplus \R^{n-4}}$ \quad when $n > 4$,
	         \end{itemize}
	         where $\lambda \in \R$. Note that, these algebras are exactly ones listed in {\bf subcase 1.2} of Theorem \ref{thm1}.

	         \item\label{case2-lem4} {\bf The second case of Lemma \ref{lem4}: One of $a_{X_3}$ and $a_{X_4}$ has real eigenvalues, 
	         the other has complex eigenvalues.}
	         
	         Without loss of generality, we assume that $a_{X_4}$ has two real eigenvalues (distinct or coincident), and $a_{X_3}$ has 
	         two conjugate complex eigenvalues. Therefore, we have
	         \[
	            a_{X_4} \in \left\{ \begin{bmatrix} 1 & 0 \\ 0 & \lambda \end{bmatrix}, \begin{bmatrix} \lambda & 1 \\ 0 & \lambda \end{bmatrix} \, \Big| \,
	            \lambda \in \R \right\}, \quad
	            a_{X_3} = \begin{bmatrix} a_{11} & a_{12} \\ a_{21} & a_{22} \end{bmatrix}.
	         \]
	         Because $a_{X_3}$ has two conjugate complex eigenvalues, its characteristic polynomial
	         \[
	            p(x) = x^2 - (a_{11}+a_{22})x + (a_{11}a_{22} - a_{12}a_{21})
	         \]
	         has the discriminant $\Delta = (a_{11}+a_{22})^2 - 4(a_{11}a_{22} - a_{12}a_{21}) = (a_{11} - a_{22})^2 + 4a_{12}a_{21} < 0$. In particular 
	         \begin{equation}\label{ineq-lem4}
	            a_{12}a_{21} < 0 \quad \text{and} \quad
	            \det \left(a_{X_3}\right) = a_{11}a_{22} - a_{12}a_{21} > 0.
	         \end{equation}
	         According to the types of $a_{X_4}$, we have the following subcases.
	         \begin{itemize}
	            \item[$\bullet$] If $a_{X_4} = \begin{bmatrix} \lambda & 1 \\ 0 & \lambda \end{bmatrix} \, (\lambda \in \R)$, the commutativity 
	            of $a_{X_3}$ with $a_{X_4}$ implies $ a_{21} = 0$ and $a_{11} = a_{22}$ which contradicts to the inequations \eqref{ineq-lem4}. 
	            Therefore, this subcase is impossible.
	            
	            \item[$\bullet$] If $a_{X_4} = \begin{bmatrix} 1 & 0 \\ 0 & \lambda \end{bmatrix} \, (\lambda \in \R)$,  the commutativity of 
	            $a_{X_3}$ with $a_{X_4}$ implies $\lambda = 1$ or $a_{12} = a_{21} = 0$. Combining these results with the inequations 
	            \eqref{ineq-lem4}, we get $\lambda = 1$, i.e. $a_{X_4} = \begin{bmatrix} 1 & 0 \\ 0 & 1 \end{bmatrix}$.
	         \end{itemize}
	         
	         Now, by the same way as the first case \ref{case1-lem4} in the proof of Lemma \ref{lem4} above, we get
	         \[
	            [X_3, X_4] = [X_3, X_k] = [X_4, X_k] = [X_i, X_j] = 0, \; \text{for all } k \geqslant 5 \text{ and } 5 \leqslant i < j \leqslant n.
	         \]
	         Replacing $X_3$ by $\frac{1}{\sqrt{\det a_{X_3}}}X_3$ we convert the value of $\det \left(a_{X_3}\right)$ to 1. 
	         Hence, without loss of generality, we can now assume that
	         \[
	            a_{X_3} = \begin{bmatrix} a_{11} & a_{12} \\ a_{21} & a_{22} \end{bmatrix}, \quad 
	            \text{where $\det \left(a_{X_3}\right) = a_{11}a_{22} - a_{12}a_{21} = 1$ and $a_{12}a_{21} < 0$}.
	         \]
	         It follows easily from a well-known fundamental result in Linear Algebra that there exists a basis $(X'_1, X'_2)$ of $\G^1$ in 
	         which $a_{X_3}$ achieves its \emph{real Jordan normal form} as follows
	         \[
	            a_{X_3} = \begin{bmatrix} \cos \varphi & -\sin \varphi \\ \sin \varphi & \cos \varphi \end{bmatrix}, \quad \varphi \in (0, \pi),
	         \]
	         while $a_{X_4}$ is obviously unchanged because it is the identity operator. Finally, by setting 
	         \[
	            X'_3 = - \frac{1}{\sin \varphi}X_3 + (\cot \varphi)X_4,
	         \]
	         we reduce $a_{X_3}$ to the new form $a_{X_3} = \begin{bmatrix} 0 & 1 \\ -1 & 0 \end{bmatrix}$. Recall that, here, 
	         $a_{X_4} = \begin{bmatrix} 1 & 0 \\ 0 & 1 \end{bmatrix}$. It is well-known that $\s \{X_1, X_2, X_3, X_4\}$ with such 
	         $a_{X_3}$ and $a_{X_4}$ is the \emph{complex affine Lie algebra $\af (\C)$}.
	         
	         To summarize, in this case, we get 
	         \begin{itemize}
	            \item[$\bullet$]$\boldsymbol{\G \cong \af (\C)}$ \qquad \qquad \, when $n=4$.
	            \item[$\bullet$]$\boldsymbol{\G \cong \af (\C) \oplus \R^{n-4}}$ \, when $n > 4$.
	         \end{itemize}
	         Note that, the complex affine Lie algebra $\af (\C)$ is exactly the one listed in {\bf subcase 1.2 (iv)} of Theorem \ref{thm1}.

	         \item\label{case3-lem4} {\bf The third case of Lemma \ref{lem4}: $a_{X_3}$ and $a_{X_4}$ have complex eigenvalues.}
	         
	         In this case, according to Lie's Theorem, $a_{X_3}$ and $a_{X_4}$ must have at least one common complex eigenvector $T$ 
	         of the following form
	         \[
	            T=(a+ic)X_1+(b+id)X_2=X'_1+iX'_2 \in \left(\G^1\right)^\C = \G^1 \oplus i\G^1,
	         \]
	         where $i$ is the {\it imaginary unit}, $\left(\G^1\right)^\C$ is the \emph{complexification of $\G^1$}, and
	         \[
	            X'_1 = aX_1 + bX_2, \quad X'_2 = cX_1+ dX_2, \quad (a, b, c, d \in \R).
	         \]
	         
	         Suppose that $\alpha_{X_3} + i\beta_{X_3}$ (resp. $\alpha_{X_4} + i\beta_{X_4}$) is the complex eigenvalue corresponding to 
	         the eigenvector $T$ of $ a_{X_3}$ (resp. $a_{X_4}$). The equation $a_{X_3} (T) = (\alpha_{X_3} + i\beta_{X_3})T$ shows that
	         \begin{equation}\label{eq3-lem4}
	            \begin{cases}
	               [X_3, X'_1] = \alpha_{X_3}X'_1 - \beta_{X_3}X'_2, \\
	               [X_3, X'_2] = \beta_{X_3}X'_1 + \alpha_{X_3}X'_2.
	            \end{cases}
	         \end{equation}
	         Similarly, the equation $a_{X_4}(T) = (\alpha_{X_4} + i\beta_{X_4})T$ shows that
	         \begin{equation}\label{eq4-lem4}
	            \begin{cases}
	               [X_4, X'_1] = \alpha_{X_4}X'_1 - \beta_{X_4}X'_2, \\
	               [X_4, X'_2] = \beta_{X_4}X'_1 + \alpha_{X_4}X'_2.
	            \end{cases}
	         \end{equation}
	         
	         It is not hard to verify that $(X'_1, -X'_2)$ is a basis of $\G^1$. Taking account of equations \eqref{eq3-lem4} and 
	         \eqref{eq4-lem4}, the matrices of $a_{X_3}$ and $a_{X_4}$ with respect to this basis are given as follows
	         \[
	            a_{X_3} = \begin{bmatrix} \alpha_{X_3} & -\beta_{X_3} \\ \beta_{X_3} & \alpha_{X_3} \end{bmatrix} \quad \text{and} \quad
	            a_{X_4} = \begin{bmatrix} \alpha_{X_4} & -\beta_{X_4} \\ \beta_{X_4} & \alpha_{X_4} \end{bmatrix}.
	         \]
	         
	         Since $ a_{X_3}$ and $a_{X_4}$ are linearly independent, the matrix $\begin{bmatrix} \alpha_{X_3} & \alpha_{X_4} \\ 
	         \beta_{X_3} & \beta_{X_4} \end{bmatrix}$ is invertible. Now we set
	         \[
	            \begin{bmatrix} x & y \\ z & t \end{bmatrix} =
	            \begin{bmatrix} \alpha_{X_3} & \alpha_{X_4} \\ \beta_{X_3} & \beta_{X_4} \end{bmatrix}^{-1}
	            \begin{bmatrix} 1 & 0 \\ 0 & -1 \end{bmatrix}.
	         \]
	         It  means that
	         \begin{equation}\label{eq5-lem4}
	            \begin{cases}
	               \alpha_{X_3}x + \alpha_{X_4}z = 1,\\
	               \beta_{X_3}x + \beta_{X_4}z = 0,
	            \end{cases}
	            \quad \text{and} \quad
	            \begin{cases}
	               \alpha_{X_3}y + \alpha_{X_4}t = 0,\\
	               \beta_{X_3}y + \beta_{X_4}t = -1.
	            \end{cases}
	         \end{equation}
	         
	         Then by using the transformation
	         \[
	            X'_3 = yX_3 + tX_4, \quad X'_4 = xX_3 + zX_4,
	         \]
	         and taking account of equations \eqref{eq5-lem4}, we convert $a_{X_3}$ and $a_{X_4}$ to the following forms
	         \[
	            a_{X_3} = \begin{bmatrix} 0 & 1 \\ -1 & 0 \end{bmatrix} \quad \text{and} \quad
	            a_{X_4} = \begin{bmatrix} 1 & 0 \\ 0 & 1 \end{bmatrix}.
	         \]
	         Therefore, Case \ref{case3-lem4} returns to Case \ref{case2-lem4} above, i.e. we get, once again, $\G \cong \af(\C)$ when $n = 4$ and $\G \cong \af(\C) \oplus \R^{n-4}$ when $n > 4$.
	      \end{enumerate}
	      The proof of Lemma \ref{lem4} is complete.
	   \end{proof}

	   It is obvious that Lemmas \ref{lem1}, \ref{lem2}, \ref{lem3} and \ref{lem4} have covered all mutually-exclusive possibilities of $\A_{\G} \neq 0$. 
	   All algebras listed in Theorem \ref{thm1} have appeared in the proof. Hence, the proof of Theorem \ref{thm1} is complete. 
	   For convenience, we summarize Theorem \ref{thm1} case by case in Table \ref{tab1} by which we can see clearly a new complete 
	   classification of non 2-step nilpotent Lie algebras in \Li. 
	   \begin{table}[!h]
	      \caption{A new classification of non 2-step nilpotent Lie algebras in {\Li}}\label{tab1} 
	      \begin{tabular}{l l l l l l}
	         \hline\noalign{\smallskip}
	            Lemmas & Cases & Indecomposable Types & Conditions & Decomposable Types & Comments \\
	         \noalign{\smallskip}\hline\noalign{\smallskip}
	         \multirow{3}{*}{\ref{lem2}} & \ref{case1-lem2} & $\G_{3,2.1(\lambda)}$ & $\lambda \in \R \setminus \{0\}$ & 
	         $\R^{n-3} \oplus \G_{3,2.1(\lambda)}$\\
	         & \ref{case2-lem2} & $\G_{3,2.2}$ & & $\R^{n-3} \oplus \G_{3,2.2}$ \\
	         & \ref{case3-lem2} & $\G_{3,2.3(\varphi)}$ & $\varphi \in (0, \pi)$ & $\R^{n-3} \oplus \G_{3,2.3(\varphi)}$ \\ 
	         \noalign{\smallskip}\hline\noalign{\smallskip}
	         \multirow{6}{*}{\ref{lem3}} & \ref{case1B-lem3} & $\G_{4,2.1}$ & & $\R^{n-4} \oplus \G_{4,2.1}$\\
	         & \ref{case2A-lem3} & $\G_{4,2.2}$ & & $\R^{n-4} \oplus \G_{4,2.2}$ & 3-step nilpotent \\
	         \noalign{\smallskip}\cline{2-6}\noalign{\smallskip}
	         & \ref{case1A-lem3} \& \ref{case1B-lem3} & & & $\R^{n-2m-3} \oplus \af (\R) \oplus \h_{2m+1}$ \\
	         \noalign{\smallskip}\cline{2-6}\noalign{\smallskip}
	         & \ref{case2B-lem3} & $\G_{5+2k,2}$ & & $\R^{n-5-2k} \oplus \G_{5+2k,2}$ & 3-step nilpotent \\
	         \noalign{\smallskip}\cline{2-6}\noalign{\smallskip}
	         & \ref{case1B-lem3} & $\G_{6+2k,2.1}$ & & $\R^{n-6-2k} \oplus \G_{6+2k,2.1}$ & \\
	         & \ref{case2B-lem3} & $\G_{6+2k,2.2}$ & & $\R^{n-6-2k} \oplus \G_{6+2k,2.2}$ & 3-step nilpotent \\ \hline
	         \multirow{4}{*}{\ref{lem4}} & \ref{case1A-lem4} & & & $\R^{n-4} \oplus \af (\R) \oplus \af (\R)$ \\
	         \noalign{\smallskip}\cline{2-6}\noalign{\smallskip}
	         & \ref{case1B-lem4} & $\G_{4,2.3(\lambda)}$ & $\lambda \in \R$ & $\R^{n-4} \oplus \G_{4,2.3(\lambda)}$ \\
	         & \ref{case2-lem4} \& \ref{case3-lem4} & $\G_{4,2.4} = \af (\C)$ & & $\R^{n-4} \oplus \af (\C) $\\
	         \noalign{\smallskip}\hline
	      \end{tabular}
	   \end{table}
	  Furthermore, Table \ref{tab1} also gives us the following consequence.
	   
	   \begin{cor}\label{cor2}
	      There is no real indecomposable non-nilpotent solvable Lie algebra of odd dimension $n \geqslant 5$ which has 
	      2-dimensional derived ideal.
	   \end{cor}

	   \subsection{\bf Proof of Theorem \ref{thm2}}
	   
	   By assumption, $\dim \A_\G = 1$, therefore we can choose an element $Z$ such that $a_Z \neq 0$ and $\A_\G = \s \{a_Z\}$. 
	   Now, we add $Y \in \G \setminus \G^1$ to get a basis $(X_1, X_2, \ldots, X_{n-2}, Y, Z)$ of $\G$. Then, there exists $y$ such that 
	   $a_Y = ya_Z$ and we can eliminate $a_Y$ by transformation $Y' = Y - yZ$. Therefore, without loss of generality, we can always 
	   assume that $a_Y = 0$. Recall that, by definition, $\G^1 = \s \{[Z, X_1], [Z, X_2], \ldots, [Z, X_{n-2}], [Z, Y]\}$. We first observe that  
	  \begin{itemize}
	  	\item[$\bullet$] $n-2 = \dim \G^1 = \dim \I (a_Z) + \dim \Ker (a_Z)$. If $\ra (a_Z) = n-3$ then $\dim \Ker (a_Z) = 1$.
	  	\item[$\bullet$] $n-3 \leqslant \dim \I (a_Z) = \ra (a_Z) = \ra \left([Z, X_1], [Z, X_2], \ldots, [Z, X_{n-2}]\right) \leqslant n-2$.
	  \end{itemize}
	  
	  There are two mutually-exclusive cases as follows.

	   \begin{enumerate}[\bf A.]
	      \item {\bf The first case of Theorem \ref{thm2}: $[Y, Z] = 0$ or $\ra (a_Z) = n-2$.}

		\begin{enumerate}
			\item[\bf A1.] {\bf The first subcase of Case A in Theorem \ref{thm2}: $[Y, Z] = 0$.}
			
			In this subcase, it is easily seen that $\G$ is decomposable. Namely, $\G = \R.Y \oplus \bar{\G}$, where 
			$\bar{\G} = \s \{X_1, X_2, \ldots, X_{n-2}, Z\} \in \mathrm{Lie} \left(n-1, n-2\right)$ with the derived ideal $\left[\bar{\G},\bar{\G}\right] = \G^1$ 
			is 1-codimensional.
			
			\item[\bf A2.] {\bf The second subcase of Case A in Theorem \ref{thm2}: $\ra (a_Z) = n-2$.}
	         
	         That means $\G^1 = \I (a_Z)$ and $a_Z$ is non-singular. Assume that 
	         \[
	            [Z, Y] = y_1 X_1 + y_2 X_2 + \ldots + y_{n-2} X_{n-2}.
	         \]
	         Then by setting
	         \[
	            Y' = Y - (y'_1 X_1 + y'_2 X_2 + \ldots + y'_{n-2} X_{n-2}) \quad \text{where} \quad
	            \begin{bmatrix} y'_1 \\ \vdots \\ y'_{n-2} \end{bmatrix} = a_{Z}^{-1} \begin{bmatrix} y_1 \\ \vdots \\ y_{n-2} \end{bmatrix}, 
	         \]
	         we get $[Z, Y'] = 0$. Therefore this subcase converts to case {\bf A1} which is considered above.  			
	      \end{enumerate}				      

	      \item\label{caseB-thm2} {\bf The second case of Theorem \ref{thm2}: $[Y, Z] \neq 0$ and $\ra (a_Z) = n-3$.}
	      
	      In this case, we have two following subcases. 

	      \begin{enumerate}[\bf B1.]
	         \item {\bf The first subcase of Case B in Theorem \ref{thm2}: $\I (a_Z) \cap \Ker (a_Z) = \{0\}$.}
      	   
      	   In this subcase, renumbering $X_1, X_2, \ldots, X_{n-2}$, if necessary, we can always assume that $([Z, X_1], \ldots, [Z, X_{n-3}])$ 
      	   is the basis of $\I (a_Z)$ and, of course, $[Z, Y]\in \G^1 \setminus \I (a_Z)$. It is easily seen that $\G^1 = \I (a_Z) \oplus \Ker (a_Z)$. 
      	   Therefore, we can set $[Z, Y] = a_Z(U) + V$ for some $U \in \G^1$ and $V \in \Ker (a_Z)$. Now, by replacing $X'_{n-2} = V$ and 
      	   $Y' = Y - U$, we obtain new bases $(X_1, \ldots, X_{n-3}, X'_{n-2})$ and $(X_1, \ldots, X_{n-3}, X'_{n-2}, Y',Z)$ of $\G^1$ and $\G$, 
      	   respectively. For these bases, we have $[Z, Y'] = X'_{n-2}$ and $a_Z$ is converted to the form as follows
      	   \[
      	      a_Z = \bar{A} = 
      	         \begin{bmatrix}
      	            a_{11} & \ldots & a_{1(n-3)} & 0 \\
      	            \vdots & \ddots & \vdots & \vdots \\
      	            a_{(n-3)1} & \ldots & a_{(n-3)(n-3)} & 0 \\
      	            0 & \ldots & 0 & 0
      	         \end{bmatrix} 
      	      = \begin{bmatrix} A & 0 \\ 0 & 0 \end{bmatrix},
      	   \]
      	   with $A = (a_{ij}) \in \GL_{n-3}(\R)$. Therefore, without loss of generality, we can always assume that $\G$ has a basis 
      	   $(X_1, X_2, \dots, X_{n-2}, Y, Z)$ such that $\G^1 = \s \{X_1, X_2, \ldots, X_{n-2}\} \cong \R^{n-2}$ and $[Z, Y] = X_{n-2}$. 
      	   Moreover, the Lie structure of $\G$ is completely determined by the matrix $\bar{A}$. 
      	   
      	   \item {\bf The second subcase of Case B in Theorem \ref{thm2}: $\I (a_Z) \cap \Ker (a_Z) \neq \{0\}$.}
      	   
      	   In this case, $\Ker (a_Z) \subset \I (a_Z)$ because $\dim \Ker (a_Z) = \dim \G^1 - \dim \I (a_Z) = 1$. We choose 
      	   $X'_1 \in \Ker (a_Z) \setminus \{0\}$ and add $X'_2, \ldots, X'_{n-3}$ to get a new basis $(X'_1, X'_2, \dots, X'_{n-3})$ of $\I (a_Z)$. 
      	   Since $[Z, Y] \in \G^1 \setminus \I (a_Z)$, we can set $X'_{n-2} = [Z, Y]$ and obtain the new basis $(X'_1, X'_2, \dots, X'_{n-2})$ 
      	   of $\G^1$. Then, $[Z, Y] = X'_{n-2}$ and $a_Z$ is converted to the form as follows
      	   \[
      	      a_Z = \bar{A} = 
      	         \begin{bmatrix}
      	            0 & a_{11} & \cdots & a_{1(n-3)} \\
      	            \vdots & \vdots & \ddots & \vdots \\
      	            0 & a_{(n-3)1} & \cdots & a_{(n-3)(n-3)} \\
      	            0 & 0 & \cdots & 0
      	         \end{bmatrix}
      	      = \begin{bmatrix} 0 & A \\ 0 & 0\end{bmatrix},
      	   \]
      	   with $A = (a_{ij}) \in \GL_{n-3}(\R)$. Once again, we can assume, without loss of generality, that $\G$ has one basis 
      	   $(X_1, X_2, \dots, X_{n-2}, Y, Z)$ such that $\G^1 = \s \{X_1, X_2, \ldots, X_{n-2}\} \cong \R^{n-2}$ and $[Z, Y] = X_{n-2}$. 
      	   Here, the Lie structure of $\G$ is also completely determined by $\bar{A}$.
      	\end{enumerate}
      	To summarize, case \ref{caseB-thm2} shows that $\G$ admits a basis $(X_1, X_2, \dots, X_{n-2}, Y, Z)$ such that 
      	$\G^1 = \s \{X_1, X_2, \ldots, X_{n-2}\} \cong \R^{n-2},\,[Z, Y] = X_{n-2}$ and
      	\[
      	   a_Z = \bar{A} \in \Bigg\{\begin{bmatrix} A & 0 \\ 0 & 0\end{bmatrix}, \begin{bmatrix} 0 & A \\ 0 & 0\end{bmatrix}\Bigg\} \quad
      	   \text{where $A \in \GL_{n-3}(\R)$}.
      	\]
      	We emphasize that, the Lie structure of $\G$ is completely determined by $\bar{A}$. Therefore, $\bar{A}$ is called 
      	the \emph{structure matrix} of $\G$ and we denote $\G$ by $\G_{\bar{A}}$.
      	
      	Now, we consider another $\G_{\bar{B}}$ with the structure matrix
      	\[
      	   a_Z = \bar{B} \in \Bigg\{\begin{bmatrix} B & 0 \\ 0 & 0\end{bmatrix}, \begin{bmatrix} 0 & B \\ 0 & 0\end{bmatrix}\Bigg\} \quad
      	   \text{for some $B \in \GL_{n-3}(\R)$}.
      	\]
      	It means that
      	\begin{itemize}
      	   \item[$\bullet$] $\G_{\bar{B}} = \s \{X_1, X_2, \ldots, X_{n-2}, Y, Z\}$ whose $\G^1_{\bar{B}} = \s \{X_1, X_2, \ldots, X_{n-2}\} \cong \R^{n-2}$.
      	   \item[$\bullet$] $[Z, Y] = X_{n-2}$ and $a_Z = \bar{B}$.
      	\end{itemize}
      	We will prove that $\G_{\bar{A}} \cong \G_{\bar{B}}$ if and only if $\bar{A} \sim_{p} \bar{B}$.
      	
      	{\bf Proof of ($\Longrightarrow$)}
      	
      	Suppose that $\G_{\bar{A}} \cong \G_{\bar{B}}$ and $f: \G_{\bar{A}} \to \G_{\bar{B}}$ is a Lie isomorphism. Let $M_f \in\GL_n(\R)$ 
      	is the matrix of $f$ with respect to the basis $(X_1, X_2, \ldots, X_{n-2}, Y, Z)$. Because $\G^1 = \s (X_1, X_2, \ldots, X_{n-2})$ is 
      	invariant under $f$, the matrix $M_f$ must be given as follows
      	\[
      	   M_f =
      	      \begin{bmatrix}
      	         c_{11} & \cdots & c_{1(n-2)} & c_{1(n-1)} & c_{1n}\\
      	         \vdots & \ddots & \vdots & \vdots & \vdots \\
      	         c_{(n-2)1} & \cdots & c_{(n-2)(n - 2)} & c_{(n-2)(n-1)} & c_{(n-2)n} \\
      	         0 & \cdots & 0 & y_1 & y_2 \\
      	         0 & \cdots & 0 & z_1 & z_2 
      	      \end{bmatrix}
      	   = \begin{bmatrix} C & \ast \\ 0 & D \end{bmatrix},
      	\]
      	where 
      	\[
      	   C = \begin{bmatrix} c_{11} & \cdots & c_{1(n-2)} \\ \vdots & \ddots & \vdots \\ c_{(n-2)1} & \cdots & c_{(n-2)(n - 2)} \end{bmatrix} 
      	   \in \M_{n-2}(\R), \quad
      	   D = \begin{bmatrix} y_1 & y_2 \\ z_1 & z_2 \end{bmatrix} \in \M_{2}(\R),
      	\]
      	and the asterisk denotes the $(n-2) \times 2$ matrix which is at the right upper corner of $M_f$. Since $f$ is an isomorphism, 
      	$\det M_f = \det C \cdot \det D \neq 0$. Therefore $C \in \GL_{n-2}(\R)$ and $D \in \GL_{2}(\R)$. In this proof, we will denote 
      	the Lie brackets of $\G_{\bar{A}}$ and $\G_{\bar{B}}$ by $[\cdot, \cdot]_{\bar{A}}$ and $[\cdot, \cdot]_{\bar{B}}$, respectively.
      	
      	Note that, by fixing basis $(X_1, X_2, \ldots, X_{n-2})$ in $\G^1$, for the sake of convenience, we can identify $\G^1$ with $\R^{n-2}$ 
      	in the following sense.
      	
      	\begin{itemize}
      	   \item[$\bullet$] The basis $(X_1, X_2, \ldots, X_{n-2})$ in $\G^1$ is identified with the canonical one of $\R^{n-2}$, i.e.
      	   \[
      	      X_1 \equiv e_1 := \begin{bmatrix} 1 \\ 0 \\ \vdots \\ 0 \end{bmatrix}, 
      	      X_2 \equiv e_2 := \begin{bmatrix} 0 \\ 1 \\ \vdots \\ 0 \end{bmatrix}, \ldots, 
      	      X_{n-2} \equiv e_{n-2} := \begin{bmatrix} 0 \\ 0 \\ \vdots \\ 1 \end{bmatrix} \in \R^{n-2}.
      	   \]
      	   \item[$\bullet$] For $X_j \,(j = 1, 2, \dots, n-2)$ and every vector $v = x_1X_1 + x_2X_2 + \cdots + x_{n-2}X_{n-2} \in \G^1$ 
      	   we have the following identities
      	   \[
      	      [Z, X_j]_{\bar{A}} \equiv \bar{A}e_j, f(X_j) \equiv Ce_j, v \equiv \begin{bmatrix} x_1 \\ x_2 \\ \vdots \\ x_{n-2} \end{bmatrix}, 
      	      [Z, v]_{\bar{B}} \equiv \bar{B}\begin{bmatrix} x_1 \\ x_2 \\ \vdots \\ x_{n-2} \end{bmatrix}, 
      	      f(v) \equiv C \begin{bmatrix} x_1 \\ x_2 \\ \vdots \\ x_{n-2} \end{bmatrix}.
      	   \]
      	\end{itemize}
      	
      	Note that $a_Y = 0 \neq a_Z$, in particular, $\bar{B}C \neq 0$. It means that there exists at least one $j \in \{1, 2, \ldots, n-2 \}$ 
      	such that ${\bar{B}}Ce_j \neq 0$. Therefore, we have
      	\[
      	   \begin{array}{l l}
      	      &  0 = f(0) = f\left([Y, X_j]_{\bar{A}}\right) = [f(Y), f(X_j)]_{\bar{B}}; \quad j = 1, 2, \dots, n-2 \\
      	      \Leftrightarrow & 0 = \left[\sum \limits_{k=1}^{n-2} {c_{k,n-1}X_k + y_1Y + z_1Z}, Ce_j \right]_{\bar{B}}; \quad j = 1, 2, \dots, n-2 \\
      	      \Leftrightarrow & 0 = z_1{\bar{B}}Ce_j; \quad j = 1, 2, \dots, n-2 \\
      	      \Leftrightarrow & 0 = z_1.
      	   \end{array}
      	\]
      	In particular, $0 \neq \det D = y_1z_2 - z_1y_2 = y_1z_2$, i.e. $z_2 \neq 0$. By setting $c = \frac{1}{z_2}$ we get $c \neq 0$. 
      	On the other hand, because $(X_1, X_2, \ldots, X_{n-2})$ is a basis of $\G^1$, we have
      	\[
      	   \begin{array}{l l}
      	      & f([Z, X_j]_{\bar{A}}) = [f(Z), f(X_j)]_{\bar{B}}; \quad j = 1, 2, \dots, n-2 \\
      	      \Leftrightarrow & f(\bar{A}e_j)= \left[\sum \limits_{k=1}^{n-2} {c_{kn}X_k + y_2Y + z_2Z}, \, Ce_j \right]_{\bar{B}}; \quad j = 1, 2, \dots, n-2 \\
      	      \Leftrightarrow & C\bar{A}e_j   = [z_2Z, Ce_j]_{\bar{B}}; \quad j = 1, 2, \dots, n-2 \\
      	      \Leftrightarrow  & C\bar{A}e_j   = z_2 \bar{B}Ce_j; \quad j = 1, 2, \dots, n-2 \\
      	      \Leftrightarrow & cC\bar{A}e_j = \bar{B}Ce_j; \quad j = 1, 2, \dots, n-2 \\
      	      \Leftrightarrow & cC\bar{A} = \bar{B}C \\
      	      \Leftrightarrow & c\bar{A} = C^{-1}\bar{B}C \\
      	      \Leftrightarrow & \bar{A} \sim_{p} \bar{B}.
      	   \end{array}
      	\]
      	
      	{\bf Proof of $(\Longleftarrow)$}
      	
      	Suppose that $\bar{A} \sim_{p} \bar{B}$, i.e. there exists a non-zero real number $c$ and an invertible $(n-2)$-squared matrix $C$ 
      	such that $c\bar{A} = C^{-1}\bar{B}C$. We consider the map $f: \G_{\bar{A}} \longrightarrow \G_{\bar{B}}$ which is defined in 
      	the basis $(X_1, X_2, \ldots, X_{n-2}, Y, Z)$ by the following $n$-squared matrix 
      	\[
      	   M_f = \begin{bmatrix} C & 0 & 0 \\ 0& c & 0 \\ 0 & 0 & \frac{1}{c} \end{bmatrix}.
      	\]
      	It is easy to check that $f$ is a Lie isomorphism. Therefore $\G_{\bar{A}} \cong \G_{\bar{B}}$. 
      	The proof of Theorem \ref{thm2} is completed.
	   \end{enumerate}
	   
	\section{Comments, remarks and illustrations}\label{sec:5}
	
	   In this section, we will give some post-hoc analyses of Theorems \ref{thm1} and \ref{thm2} based on comparisons with previous 
	   classifications of some subclasses of solvable Lie algebras.

	   \subsection{\bf {\Li} in comparisons with Sch\"obel \cite{Sch93} and Jannisse \cite{Jan10}}
	      
	      First of all, we give in this subsection a correspondence between our classification and the works of Sch\"obel \cite{Sch93} 
	      and Jannisse \cite{Jan10} in same problem and note the missing in their works.
	      
	      \begin{enumerate}
	         \item In 1993 Sch\"obel \cite{Sch93} classified $\mathrm{Lie} \left(n, k\right)$ with $k \in \{1, 2, 3\}$. More concretely, 
	         he gave a partial classification of {\Li} when $\dim \left(L^{(1)} \cap C(L)\right) \leqslant 1$, where $L^{(1)}$ (resp. $C(L)$) 
	         is the derived algebra (resp. the center) of the considered Lie algebra $L$, and this result corresponds to this paper. 
	         However, its detailed proof was not clear enough because it cited to a preprint article which cannot be found in J. Math. Phys. 
	         Furthermore, there is a missing in this classification as follows.
	         
	         He began with a without-proof Lemma in \cite[p.\ 177]{Sch93} asserted that $L$ has a 4-dimensional subalgebra $S$ 
	         with $\dim S^{(1)} = 2$. Then nine Lie algebras $S$ denoted by $a1$, $a2$, $a3$, $a4 \, (p \neq 0)$, $b5$, $b6$, $b7$, $c8$ 
	         and $c9$ were specified. Among them, $a1$, $a2$, $a3$, $a4$ and $b5$ are decomposable. Moreover, $c8$ and $c9$ 
	         have $\dim \left( S^{(1)} \cap C(S)\right) = 1$, otherwise $S^{(1)} \cap C(S) = \{0\}$. Now the partial classification of {\Li} in 
	         Sch\"obel \cite[Theorem 2]{Sch93} contains two subcases as follows.
	         
	         \begin{enumerate}
	            \item If $L^{(1)} \cap C(L) = \{0\}$ then $L = \R^{n-4} \oplus S$, $C(L) = \R^{n-4} \oplus C(S)$, where $S$ is a 4-dimensional 
	            real Lie algebra with 2-dimensional derived algebra and $S^{(1)} \cap C(S) = \{0\}$. Therefore, the desired classification 
	            amounts to the above classification of $S$ which $S^{(1)} \cap C(S) = \{0\}$. It follows that we have the Lie algebras 
	            $\R^{n-4} \oplus ai \, (i = 1, 2, 3, 4)$ or $\R^{n-4} \oplus bj \, (j = 1, 2, 3)$.
	            
	            \item If $\dim \left(L^{(1)} \cap C(L)\right) = 1$ then $L = I_{n-1} \oplus_{\ad_Z} L_1$ is the semi-direct sum of an 
	            $(n-1)$-dimensional ideal $I_{n-1} = \s \{X_1, \ldots, X_{n-1}\}$, i.e. $I_{n-1}$ belongs to $\mathrm{Lie} \left(n-1, 1\right)$, 
	            and an 1-dimensional subalgebra $L_1 = \s \{Z\}$. Thus the classification of {\Li} in this subcase is reduced to the classification 
	            of $\mathrm{Lie} \left(n-1, 1\right)$ which had been known, and the adjoint operator $\ad_Z$ has one of four forms: 
	            $\varphi_1 (Z)$, $\varphi_2 (Z)$, $\varphi_3 (Z)$, and $\varphi_4 (Z)$.
	            
	            From usual abbreviations, $\varepsilon_i$ should be 0 or 1 even though there is no condition here, and thus $\varphi_1(Z)$ 
	            is a special case of $\varphi_2(Z)$. It can verify that we will have in this subcase the following Lie algebras
	            \begin{itemize}
	               \item $\R^{n-4} \oplus ci \, (i=8,9)$ or $\R^{n-2k-5} \oplus \G_{5+2k,2}$ or $\R^{n-2k-6} \oplus \G_{6+2k,2.1}$ if $\ad_Z = \varphi_2 (Z)$.
	               \item $\R^{n-2m-3} \oplus \af (\R) \oplus \h_{2m+1}$ if $\ad_Z = \varphi_3 (Z)$.
	               \item $\R^{n-2k-6} \oplus \G_{6+2k,2.2}$ if $\ad_Z = \varphi_4 (Z)$.
	            \end{itemize}
	         \end{enumerate}

	         To summarize, these results are shown in Table \ref{tab2} below.
	         \begin{table}[!h]
	            \caption{Correspondence between our classification and Sch\"obel \cite[Theorem 2]{Sch93}}\label{tab2}
	            \begin{tabular}{|l||c|c|c|c|c|c|c|c|c|c|c|c|}
	               \hline \multirow{2}{*}{\backslashbox{Our types}{Sch\"obel \cite{Jan10}}} & \multirow{2}{*}{$a1$} & \multirow{2}{*}{$a2$} & 
	               \multirow{2}{*}{$a3$} & $a4,$ & $a4,$ & $a4,$ & \multirow{2}{*}{$b5$} & \multirow{2}{*}{$b6$} & \multirow{2}{*}{$b7$} & 
	               \multirow{2}{*}{$\varphi_2$} & \multirow{2}{*}{$\varphi_3$} & \multirow{2}{*}{$\varphi_4$} \\
	               & & & & $p < -\frac{1}{4}$ & $p = -\frac{1}{4}$ & $p > -\frac{1}{4}$ & & & & & & \\ \hline
	               \hline $\R^{n-3} \oplus \G_{3,2.1(-1)}$ & \checkmark & & & & & & & & & & & \\
	               \hline $\R^{n-3} \oplus \G_{3,2.1(1)}$ & & & \checkmark & & & & & & & & & \\
	               \hline $\R^{n-3} \oplus \G_{3,2.1(\lambda \neq \pm 1)}$ & & & & & & \checkmark & & & & & & \\
	               \hline $\R^{n-3} \oplus \G_{3,2.2}$ & & & & & \checkmark & & & & & & & \\ 
	               \hline $\R^{n-3} \oplus \G_{3,2.3\left(\frac{\pi}{2}\right)}$ & & \checkmark & & & & & & & & & & \\
	               \hline $\R^{n-3} \oplus \G_{3,2.3\left(\varphi \neq \frac{\pi}{2}\right)}$ & & & & \checkmark & & & & & & & & \\
	               \hline $\R^{n-4} \oplus \G_{4,2.1}$ & & & & & & & & & & \checkmark & & \\
	               \hline $\R^{n-4} \oplus \G_{4,2.2}$ & & & & & & & & & & \checkmark & & \\
	               \hline $\R^{n-4} \oplus \G_{4,2.3(0)}$ & & & & & & & & & \checkmark & & & \\
	               \hline $\R^{n-4} \oplus \G_{4,2.4}$ & & & & & & & & \checkmark & & & & \\
	               \hline $\R^{n-2k-5} \oplus \G_{5+2k,2}$ & & & & & & & & & & \checkmark & & \\
	               \hline $\R^{n-2k-6} \oplus \G_{6+2k,2.1}$ & & & & & & & & & & \checkmark & & \\
	               \hline $\R^{n-2k-6} \oplus \G_{6+2k,2.2}$ & & & & & & & & & & & & \checkmark \\
	               \hline $\R^{n-4} \oplus \af(\R) \oplus \af(\R)$ & & & & & & & \checkmark & & & & & \\
	               \hline $\R^{n-2m-3} \oplus \af(\R) \oplus \h_{2m+1}$ & & & & & & & & & & & \checkmark & \\
	               \hline 
	            \end{tabular}
	         \end{table}
	         
	         It is clear that the family $\G_{4,2.3(\lambda \neq 0)}$ is missed here in dimension 4 and, of course, so is 
	         $\R^{n-4} \oplus \G_{4,2.3(\lambda \neq 0)}$.
	      
	         \item In 2010, Janisse \cite{Jan10} also classified all finite-dimensional Lie algebras $L$ over a field $\F$ whose 
	         derived algebra $L'$ has dimension 1 or 2. More precisely, a complete classification of $L$ when $\dim L' = 1$ and 
	         $\F$ is of characteristic zero in \cite[Proposition 4.1]{Jan10} coincides with the classification of $\mathrm{Lie} \left(n, 1\right)$ 
	         of Sch\"obel \cite[Theorem 1]{Sch93}. When $\dim L' = 2$, $L' = \s\{e_1, e_2\}$ is non-central and $\F$ is \emph{algebraically closed} 
	         \cite[Subsection 7.2]{Jan10}, he denotes by $A_i = \begin{bmatrix} a^1_{1i} & a^2_{2i} \\ a^1_{2i} & a^2_{2i} \end{bmatrix} \,
	         (3 \leqslant i \leqslant n)$, where $a_{ij}^k$ are the structure constants of $L$. After that, he divides into the following two cases.
	         
	         \begin{enumerate}
	            \item There exists an invertible linear combination of $A_i$. By \cite[Proposition 7.1]{Jan10} all outside basic vectors of $L'$ 
	            commute, and Lie structure of $L$ thus defined by all $A_i$ for $i \geqslant 3$. According to \cite[Theorem 7.1]{Jan10} there 
	            are non-zero $A_i$ as in the following four subcases:
	            \begin{enumerate}
	       	       \item $A_3 = \begin{bmatrix} 1 & 0 \\ 0 & a \end{bmatrix}$ with $a \neq 0$, and $A_i = 0$ for all $i > 3$.
	       	       \item $A_3 = \begin{bmatrix} 1 & 0 \\ 0 & 0 \end{bmatrix}$, $A_4 = \begin{bmatrix} 0 & 0 \\ 0 & 1 \end{bmatrix}$, 
	       	       and $A_i = 0$ for all $i > 4$.
	       	       \item $A_3 = \begin{bmatrix} 1 & 1 \\ 0 & 1 \end{bmatrix}$, $A_4 = \begin{bmatrix} 0 & 1 \\ 0 & 0 \end{bmatrix}$, 
	       	       and $A_i = 0$ for all $i > 4$.
	       	       \item $A_3 = \begin{bmatrix} 1 & 1 \\ 0 & 1 \end{bmatrix}$, and $A_i = 0$ for all $i > 3$.
	            \end{enumerate}
	            \item In case of $\dim [L, L'] = 1$, it follows from \cite[Theorem 7.3]{Jan10} that $L = \s \{e_1, e_2, \ldots, e_n\}$ and 
	            one of two following cases occurs:
	            \begin{enumerate}
	               \item $A_3 = \begin{bmatrix} 1 & 0 \\ 0 & 0 \end{bmatrix}$, all other $A_i = 0$, $[e_3, e_4] = e_2$, and all other Lie brackets 
	               are trivial or there exists $i \geqslant 5$ such that $[e_4, e_5] = [e_5, e_6] = \cdots = [e_{i-1}, e_i] = e_2$.
	               \item $A_3 = \begin{bmatrix} 0 & 1 \\ 0 & 0 \end{bmatrix}$, all other $A_i = 0$, $[e_3, e_4] = e_1$, and all other Lie brackets 
	               are trivial or there exists $i \geqslant 5$ such that $[e_4, e_5] = [e_5, e_6] = \cdots = [e_{i-1}, e_i] = e_2$.
	            \end{enumerate}
	         \end{enumerate}
	         
	         Now the following Table \ref{tab3} is a correspondence between our classification and Jannise's one over $\R$ from which, 
	         a partial classification of {\Li} follows. 
	      \begin{table}[!h]
	         \caption{Correspondence between our classification and Jannise \cite{Jan10} over $\R$}\label{tab3} 
	         \begin{tabular}{|l||c|c|c|c|c|c|c|c|}
	            \hline \backslashbox{Our types}{Jannise \cite{Jan10}} & ai & aii & aiii & aiv & bi & bii & biii & biv \\ \hline
	            \hline $\R^{n-3} \oplus \G_{3,2.1(\lambda)}$ & \checkmark & & & & & & & \\
	            \hline $\R^{n-3} \oplus \G_{3,2.2}$ & & & & \checkmark & & & & \\ 
	            \hline $\R^{n-4} \oplus \G_{4,2.1}$ & & & & & \checkmark & & & \\
	            \hline $\R^{n-4} \oplus \G_{4,2.2}$ & & & & & & & \checkmark & \\
	            \hline $\R^{n-4} \oplus \G_{4,2.3(0)}$ & & & \checkmark & & & & & \\
	            \hline $\R^{n-2k-5} \oplus \G_{5+2k,2}$ & & & & & & & & \checkmark \\
	            \hline $\R^{n-2k-6} \oplus \G_{6+2k,2.1}$ & & & & & & & & \checkmark \\
	            \hline $\R^{n-2k-6} \oplus \G_{6+2k,2.2}$ & & & & & & \checkmark & & \\
	            \hline $\R^{n-4} \oplus \af(\R) \oplus \af(\R)$ & & \checkmark & & & & & & \\
	            \hline $\R^{n-2m-3} \oplus \af(\R) \oplus \h_{2m+1}$ & & & & & & \checkmark & & \\
	            \hline 
	         \end{tabular}
	      \end{table}
	      
	      It is easy to see that the families $\R^{n-3} \oplus \G_{3,2.3(\varphi)}$ and $\R^{n-4} \oplus \G_{4,2.3(\lambda \neq 0)}$ do not appear.
	   \end{enumerate}

	   \subsection{\bf Illustrations of {\Li} in low dimensions}
	   
	      In this subsection, we will consider intersections of {\Li} with previous classifications of real indecomposable solvable Lie algebras 
	      in low dimension (not greater than 7). More precisely, we use the following classifications:
	      \begin{itemize}
	         \item[$\bullet$] Dimension 3 and 4: Mubarakzyanov \cite{Mub63a} in 1963 (cf. also Patera et al. \cite[Table I]{PSWZ76}).
	         \item[$\bullet$] Dimension 5: Nilpotent Lie algebras of Dixmier \cite{Dix58} in 1958 and solvable ones of Mubarakzyanov \cite{Mub63b} 
	         in 1963 (cf. Patera et al. \cite[Table II]{PSWZ76} for a more accessible version).
	         \item[$\bullet$] Dimension 6: Nilpotent Lie algebras of Morozov \cite{Mor58} in 1958 (cf. also Patera et al. \cite[Table III]{PSWZ76}) 
	         and solvable ones of Mubarakzyanov \cite{Mub63c} in 1963 (cf. Shabanskaya \& Thompson \cite{ST13} for a revised version) and 
	         Tukowski \cite{Tuk90} in 1990.
	         \item[$\bullet$] Dimension 7: Nilpotent Lie algebras of Gong \cite{Gon98} in 1998 and solvable ones of Parry \cite{Par07} in 2007 
	         and Hindeleh \& Thompson \cite{HT08} in 2008.
	      \end{itemize}

	      Now the intersection of {\Li} with the above classifications is given in Table \ref{tab4} below. 
	      \begin{table}[!h]
	         \caption{Intersection between {\Li} and the other classifications in low dimension}\label{tab4}
	         \begin{tabular}{|c|l|c|cc|ccc|ccc|c|}
	            \hline \rotatebox[origin=c]{90}{Dimensions} & \rotatebox[origin=c]{90}{Our types} & 
	            \rotatebox[origin=c]{90}{Mubarakzyanov \cite{Mub63a}} & \rotatebox[origin=c]{90}{Dixmier \cite{Dix58}} & 
	            \rotatebox[origin=c]{90}{Mubarakzyanov \cite{Mub63b}} & \rotatebox[origin=c]{90}{Morozov \cite{Mor58}} & 
	            \rotatebox[origin=c]{90}{Mubarakzyanov \cite{Mub63c}} & \rotatebox[origin=c]{90}{Tukowski \cite{Tuk90}} & 
	            \rotatebox[origin=c]{90}{Gong \cite{Gon98}} & \rotatebox[origin=c]{90}{Parry \cite{Par07}} & 
	            \rotatebox[origin=c]{90}{Hindeleh \& Thompson \cite{HT08}} &  \rotatebox[origin=c]{90}{Patera et al. \cite{PSWZ76}} \\ 
	            \hline \multirow{3}{*}{3} & $\G_{3,2.1(\lambda)}$ & $g_{3,3}$, $g_{3,4}$ & & & & & & & & & $A_{3,3}$, $A_{3,4}$, $A^a_{3,5}$ \\
	            \cline{2-12}  & $\G_{3,2.2}$ & $g_{3,2}$ & & & & & & & & & $A_{3,2}$ \\
	            \cline{2-12}  & $\G_{3,2.3}$ & $g_{3,5}$ & & & & & & & & & $A_{3,6}$, $A^a_{3,7}$ \\
	            \hline \multirow{4}{*}{4} & $\G_{4,2.1}$ & $g_{4,1}$ & & & & & & & & & $A_{4,1}$ \\
	            \cline{2-12}  & $\G_{4,2.2}$ & $g_{4,3}$ & & & & & & & & & $A_{4,3}$ \\
	            \cline{2-12}  & $\G_{4,2.3(0)}$ & $g_{4,9}^{p=0}$ & & & & & & & & & $A^0_{4,9}$ \\
	            \cline{2-12}  & $\G_{4,2.4}$ & $g_{4,10}$ & & & & & & & & & $A_{4,12}$ \\
	            \hline 5 & $\G_{5,2}$ & & $\g_{5,3}$ & & $L_{5,5}$ & & & & & &  $A_{5,5}$ \\
	            \hline \multirow{2}{*}{6} & $\G_{6,2.1}$ & & & & 2B.2c.12 & & & & & & $A_{6,12}$ \\
	            \cline{2-12}  & $\G_{6,2.2}$ & & & & & $g_{6,14}^{a=b=0}$ & & & & & \\
	            \hline 7 & $\G_{7,2}$ & & & & & & & (157) & & & \\
	            \hline
	         \end{tabular}
	      \end{table}
	      
	      We conclude this subsection by pointing out some misprints as well as shortcomings in previous works and correcting them.
	      \begin{enumerate}
	         \item In 1963 Mubarakzyanov \cite{Mub63a} classified 4-dimensional real solvable Lie algebras. It is easy to see from 
	         Table \ref{tab4} that the family $\G_{4,2.3(\lambda \neq 0)}$ is missed.
	         
	         \item In 1963 Mubarakzyanov \cite{Mub63b} also classified 5-dimensional real solvable Lie algebras. 
	         For a more accessible version we use Patera et al.\footnote{There are two misprints: for the Lie algebra $A_{5,8}^c$ the last 
	         Lie bracket $[e_4, e_5] = ce_4 \, (0 \neq c \leqslant b)$ should be removed; and the last Lie algebra $A_{5,40}$ in 
	         \cite[Table II]{PSWZ76} is non-solvable.} \cite[Table II]{PSWZ76}. Below we revise some errors.
	         \begin{itemize}
	            \item For Lie algebra $A^{bc}_{5,9}$ the condition $0 \neq c \leqslant b$ should be replaced by $bc \neq 0$. 
	            In fact, if $b = 0 \neq c$ then $A^{0c}_{5,9} \cong \R. e_3 \oplus \s \{e_1, e_2, e_4, e_5\}$ is decomposable.
	            
	            \item For the family $A^{ab}_{5,33}$ the condition should be $ab \neq 0$ instead of $a^2 + b^2 \neq 0$. In fact, if $a = 0 \neq b$ 
	            (resp. $a \neq 0 = b$) then $A^{0b}_{5,33} \cong \G_{3,2.1(b)} \oplus \af(\R)$ (resp. $A^{a0}_{5,33} \cong \af(\R) \oplus \G_{3,2.1(a)}$) 
	            is decomposable.
	         \end{itemize}
	         
	         \item In 1958 Morozov \cite{Mor58} classified 6-dimensional nilpotent Lie algebras over a field of characteristic zero 
	         (cf. also Patera et al. \cite[Table III]{PSWZ76} for a more accessible version). There exists an error here. The sixth family in 
	         \cite[Section 2, p.\ 168]{Mor58} (it is exactly $A^a_{6,5}$ in \cite[Table III]{PSWZ76}) has the Lie structure as follows
	         \[
	            \begin{array}{l l l l l l l l l}
	               L_6^\gamma: & &[e_1, e_3] = e_5, & [e_1, e_4] = e_6, & [e_2, e_3] = \gamma e_6, & [e_2, e_4] = e_5, && (\gamma \neq 0).
	            \end{array}
	         \]
	         We emphasize that it should be $[e_2, e_3] = -e_6$ instead of $[e_2, e_3] = \gamma e_6 \, (\gamma \neq 0)$. In fact, if 
	         $\gamma > 0$ then by transformation
	         \[
	            T =
	               \begin{bmatrix}
	                  1 & 0 & 0 & 1 & 0 & 0 \\
	                  \frac{1}{\sqrt{\gamma}} & 0 & 0 & -\frac{1}{\sqrt{\gamma}} & 0 & 0 \\
	                  0 & 1 & 0 & 0 & 1 & 0 \\
	                  0 & \sqrt{\gamma} & 0 & 0 & -\sqrt{\gamma} & 0 \\
	                  0 & 0 & 2 & 0 & 0 & 2 \\
	                  0 & 0 & 2\sqrt{\gamma} & 0 & 0 & -2\sqrt{\gamma}
	               \end{bmatrix},
	         \]
	         it reduces to $[e_1, e_2] = e_3$, $[e_4, e_5] = e_6$, and thus $L_6^{\gamma > 0} \cong \h_3 \oplus \h_3$ is decomposable. 
	         Similarly, if $\gamma < 0$ we can normalize it to $-1$ by transformation 
	         $T = \di \left(\sqrt{-\gamma}, -1, -1, \sqrt{-\gamma}, -\sqrt{-\gamma}, -\gamma\right)$, and thus $L_6^{\gamma < 0} \cong L_6^{-1}$. 
	         Moreover, by considering the operators $\ad_{e_1}$ and $\ad_{e_2}$ on the 4-dimensional maximal abelian ideal $I = \s \{e_3, e_4, e_5, e_6\}$, it follows $L_6^{\gamma > 0} \not \cong L_6^{-1}$. Therefore, we can remove absolutely the parameter $\gamma$ to get one and only one Lie algebra $L_6 = L_6^{-1}$.
	         \end{enumerate}

	   \subsection{\bf Illustrations of $\mathrm{Lie} \left(n, (n-2)C\right)$ in low dimensions}
	   
	   In this subsection, we give illustrations of Theorem \ref{thm2} in low dimensions. Namely, we give a concrete list of Lie algebras 
	   belong to $\mathrm{Lie} \left(n, (n-2)C\right)$ in dimension 4 and 5. Let us note that, here $\G$ is an $n$-dimensional solvable 
	   indecomposable Lie algebra with $\G^1 \cong \R^{n-2}$ and $\dim \A_\G = 1$.
	   \begin{enumerate}
	      \item {\bf The case $n = 4$.} In this case, $\G = \s \{X_1, X_2, Y, Z\}$ with $\G^1 = \s \{X_1, X_2\} \cong \R^2$, 
	      $[Y, Z] = X_2$, $a_Y = 0$ and
	      \[
	         a_Z = \begin{bmatrix} a & 0 \\ 0 & 0 \end{bmatrix} \, \sim_{p} \, \begin{bmatrix} 1 & 0 \\ 0 & 0 \end{bmatrix} \quad 
	         \text{or} \quad a_Z = \begin{bmatrix} 0 & a \\ 0 & 0 \end{bmatrix} \, \sim_{p} \, \begin{bmatrix} 0 & 1 \\ 0 & 0 \end{bmatrix}, \quad 
	         \text{for all $a \neq 0$}.
	      \]
	      Therefore, in this case we have two 4-dimensional Lie algebras having 2-codimensional derived algebra. Note that if $n = 4$ 
	      then $n-2 = 2$ and it is obvious that we obtain, once again, the same result as the one of {\Li} with $n = 4$, $\dim \A_\G = 1$ 
	      which is shown in Part \ref{part1-thm1} of Theorem \ref{thm1}, that are Lie algebras $\G_{4,2.1}$ and $\G_{4,2.2}$.
	      
	      \item {\bf The case $n = 5$.} In this case $\G = \s \{X_1, X_2, X_3, Y, Z\}$ with $\G^1 = \s \{X_1, X_2, X_3\} \cong \R^3$, 
	      $[Y, Z] = X_3$, $a_Y = 0$ and
	      \[
	         a_Z = \begin{bmatrix} A & 0 \\ 0 & 0 \end{bmatrix} \quad \text{or} \quad a_Z = \begin{bmatrix} 0 & A \\ 0 & 0 \end{bmatrix}, \quad 
	         \text{where $A \in \GL_2(\R)$}.
	      \]
	      \begin{itemize}
	         \item[$\bullet$] Assume that $a_Z = \begin{bmatrix} A & 0 \\ 0 & 0 \end{bmatrix}$. Then classification of $a_Z$ is reduced to 
	         classification of $A$ by proportional similar relation. By the proof of Lemma \ref{lem2}, we have three families of proportional 
	         similar classes of $A$ as follows
	         \[
	            \begin{bmatrix} 1 & 0 \\ 0 & \lambda \end{bmatrix} (\lambda \in \R \setminus \{0\}); \quad
	            \begin{bmatrix} 1 & 1 \\ 0 & 1 \end{bmatrix}; \quad
	            \begin{bmatrix} \cos \varphi &-\sin \varphi \\ \sin \varphi & \cos \varphi \end{bmatrix} (\varphi \in (0, \pi)).
	         \]
	         Thus, in this case, we have three families of Lie algebras.
	         
	         \item[$\bullet$] Assume that $a_Z = \begin{bmatrix} 0 & A \\ 0 & 0 \end{bmatrix}$. Namely, we can always set 
	         $a_Z = \begin{bmatrix} 0 & a & c \\ 0 & b & d \\ 0 & 0 & 0 \end{bmatrix}$ with $a, b, c, d \in \R$ and $ad - bc \neq 0$. 
	         Then, $p(t) = t^2(b - t)$ is the characteristic polynomial of $a_Z$ and $a_Z$ has eigenvalues $0, 0, b$. This means that
	         \begin{itemize}
	            \item[$\bullet$] $a_Z \sim_{p} \begin{bmatrix} 0 & 1 & 0 \\ 0 & 0 & 1 \\ 0 & 0 & 0 \end{bmatrix}$ when $b = 0$.
	            \item[$\bullet$] $a_Z \sim_{p} \begin{bmatrix} 0 & 1 & 0 \\ 0 & 0 & 0 \\ 0 & 0 & 1 \end{bmatrix}$ when $b \neq 0$.
	         \end{itemize}
	         Therefore, we obtain, in this case,  two families of Lie algebras.	         
	      \end{itemize}

	      To summarize, we have five families of Lie algebras in $\mathrm{Lie} \left(5, 3C\right)$. These algebras, of course, coincide with 
	      the corresponding ones of Mubarakzyanov \cite{Mub63a}. We show them in Table \ref{tab5} below.
	      \begin{table}[!h]
	      \caption{5-dimensional Lie algebras having 3-dimensional commutative derived algebra and $\dim A_\G = 1$}\label{tab5} 
	      \begin{tabular}{c c c c}
	         \hline\noalign{\smallskip}
	            Types of $a_Z$ & Cases & Mubarakzyanov \cite{Mub63a} & Patera et al. \cite[TABLE II]{PSWZ76} \\
	         \noalign{\smallskip}\hline\noalign{\smallskip}
	           \multirow{8}{*}{$\begin{bmatrix} A & 0 \\ 0 & 0 \end{bmatrix}$} & $\begin{bmatrix} 1 & 0 & 0 \\ 0 & \lambda & 0 \\ 0 & 0 & 0 \end{bmatrix}$ 
	           & $g_{5,8}$ & $A_{5,8}^c$ \\
	         \noalign{\smallskip}\cline{2-4}\noalign{\smallskip}
	           & $\begin{bmatrix} 1 & 1 & 0 \\ 0 & 1 & 0 \\ 0 & 0 & 0 \end{bmatrix}$ & $g_{5,15}^{\gamma = 0}$ & $A_{5,15}^{a = 0}$ \\
	           \noalign{\smallskip}\cline{2-4}\noalign{\smallskip}
	           & $\begin{bmatrix} \cos \varphi &-\sin \varphi & 0 \\ \sin \varphi & \cos \varphi & 0 \\ 0 & 0 & 0 \end{bmatrix}$ & $g_{5,14}$ 
	           & $A_{5,14}^p$ \\
	         \noalign{\smallskip}\hline\noalign{\smallskip}
	            \multirow{4}{*}{$\begin{bmatrix} 0 & A \\ 0 & 0 \end{bmatrix}$} & $\begin{bmatrix} 0 & 1 & 0 \\ 0 & 0 & 1 \\ 0 & 0 & 0 \end{bmatrix}$ 
	            & $g_{5,2}$ & $A_{5,2}$ \\
	         \noalign{\smallskip}\cline{2-4}\noalign{\smallskip}
	            & $\begin{bmatrix} 0 & 1 & 0 \\ 0 & 0 & 0 \\ 0 & 0 & 1 \end{bmatrix}$ & $g_{5,10}$ & $A_{5,10}$ \\
	         \noalign{\smallskip}\hline
	      \end{tabular}
	   \end{table}
	   \end{enumerate}

	\section{Concluding remark}\label{sec:6}
	
	We conclude the paper with the following remarks.
	\begin{itemize}
	   \item[$\bullet$] {\bf For the classes $\mathrm{Lie} \left(n, k\right)$ with $k$ is small}: Proposition \ref{prop3} asserts that if an 
	   $n$-dimensional real solvable Lie algebra $\G$ belongs to {\Li} then $0 \leqslant \dim A_\G \leqslant 2$. Furthermore, the class 
	   of real 2-step nilpotent Lie algebras was investigated by Eberlein \cite{Ebe03}, and it is correspondent to the case $\dim A_\G = 0$. 
	   Therefore, together with the case $\dim A_\G \in \{1, 2\}$ in Theorem \ref{thm1}, we have a new complete classification of \Li.
	   
	   Combining the result of Sch\"obel \cite{Sch93} in 1993, the classes $\mathrm{Lie} \left(n, 1\right)$ and {\Li} are classified completely. 
	   So far, we have the complete classification of $\mathrm{Lie} \left(n, 1\right)$ and \Li. From here, we can begin to attack the open problem, 
	   namely the classifying problem for $\mathrm{Lie} \left(n, k\right)$ with $k > 2$.
	   
	   \item[$\bullet$] {\bf For the classes $\mathrm{Lie} \left(n, n-k\right)$ with $k$ is small}: The classifying problem for the classes 
	   $\mathrm{Lie} \left(n, n-k\right)$, in general, is more complicated than for $\mathrm{Lie} \left(n, k\right)$. Firstly, we will restrict ourselves 
	   to the simplest case, namely we consider the class $\mathrm{Lie} \left(n, (n-k)C\right)$ containing Lie algebras 
	   $\G \in \mathrm{Lie} \left(n, n-k\right)$ such that $\G^1 = [\G, \G]$ is commutative. 
	   
	   Recall that the class $\mathrm{Lie} \left(n, (n-1)C\right)$ has classified completely in 2016 by Vu A. Le et al. \cite{VHTHT16}. 
	   When $\G$ belongs to $\mathrm{Lie} \left(n, (n-2)C\right)$, $n \geqslant 4$, then Proposition \ref{prop4} also asserts that 
	   $\dim A_\G \in \lbrace 1, 2 \rbrace$. Theorem \ref{thm2} gives the (incomplete) classification of $\mathrm{Lie} \left(n, (n-2)C\right)$ 
	   which is restricted in the case $\dim A_\G = 1$.
	   
	   \item[$\bullet$] {\bf For the classes $\mathrm{MD}(n,k)$ and $\mathrm{MD}(n,n-k)$}: In fact, Vu A. Le et al. \cite{VHTHT16} have 
	   classified the so called $\mathrm{MD}(n,1) \equiv \mathrm{Lie} \left(n, 1\right)$ and $\mathrm{MD}(n,n-1)\equiv \mathrm{Lie} \left(n,(n-1)C\right)$. 
	   Recall that $\mathrm{MD}(n,k)$ is the subclass of $\mathrm{Lie} \left(n, k\right)$ containing Lie algebras of $\mathrm{Lie} \left(n, k\right)$ 
	   such that the coadjoint orbits of corresponding Lie group are zero-dimensional or maximal dimensional. Each algebra 
	   of $\mathrm{MD}(n,k)$ is called an \emph{$\mathrm{MD}(n,k)$-algebra}.
	   
	   We emphasize that the key of the most important method in theory of representations of Lie groups and Lie algebras, i.e. 
	   the Orbit Method of Kirillov, is the coadjoint orbits. Hence, the classifying problem for $\mathrm{MD}(n,k)$ is worth to study. 
	   
	   From the results of Theorems \ref{thm1} and \ref{thm2}, it is not hard to check that
	   \begin{itemize}
	      \item In the classification of {\Li}, the families $\G_{3,2.1(\lambda)} (\lambda \in \R \setminus \{0\})$, $\G_{3,2.2}$, 
	      $\G_{3,2.3(\varphi)} (\varphi \in (0,\pi))$, $\G_{4,2.1}$, $\G_{4,2.2}$, $\af (\C)$ and, of course, their extensions by the 
	      commutative real Lie algebras, are $\mathrm{MD}(n,2)$-algebras.
	      \item Furthermore, all Lie algebras listed in Theorem \ref{thm2} are also $\mathrm{MD}(n,n-2)$-algebras.
	   \end{itemize}
	   
	   For the classes $\mathrm{Lie} \left(n,k\right)$, $\mathrm{MD}(n,k)$ with $k>2$, $\mathrm{Lie} \left(n,n-k\right)$ with $k \geqslant 1$ 
	   and $\mathrm{MD}(n,n-k)$ with $k>1$, the classifying problem is still open up to now. In the forthcoming paper, we will consider the 
	   classes $\mathrm{Lie} \left(n,n-1\right)$ and {\li} as well as $\mathrm{MD}(n,n-2)$.
\end{itemize}

\begin{aknow}
	The authors would like to take this opportunity to thank the University of Economics and Law, Vietnam National University -- Ho Chi Minh City; the University of Physical Education and Sports, Ho Chi Minh City; Can Tho University; Dong Thap University and Hoa Sen University for financial supports.
\end{aknow}



\end{document}